\documentclass[aop]{imsart}

\RequirePackage{amsthm,amsmath,amsfonts,amssymb}
\RequirePackage[numbers]{natbib}

\startlocaldefs

\usepackage{amsmath,amsthm,amsfonts,amssymb,mathrsfs,bm,graphicx,stmaryrd,hyperref,enumerate}
\usepackage[usenames]{color}
\parskip 	\smallskipamount

\newtheorem{theorem}{Theorem}[section]
\newtheorem{lemma}[theorem]{Lemma}
\newtheorem{corollary}[theorem]{Corollary}
\newtheorem{proposition}[theorem]{Proposition}

\theoremstyle{definition} 
\newtheorem{definition}[theorem]{Definition}
\newtheorem{remark}[theorem]{Remark}

\newtheorem{problem}{Problem}

\DeclareMathOperator*{\argmax}{arg\,max}
\def\ind{{\mathbf 1}}
\def\N{\mathbb{N}}
\def\Q{\mathbb{Q}}
\def\P{\mathbb{P}}
\def\Z{\mathbb{Z}}
\def\R{\mathbb{R}}
\def\S{\mathcal{S}}
\def\C{\mathcal{C}}

\def \sourav{\textcolor{black}}

\newcommand{\citDOV}{{\sf DOV}}

\def\cL{\mathcal{L}}

\def\E{\mathbb{E}}
\def \e{\varepsilon}

\newcommand{\sset}{\subset}
\newcommand{\lf}{\left}
\newcommand{\rg}{\right}

\newcommand{\eqd}{\overset{d}{=}}

\newcommand{\cvgd}{\overset{d}{\Rightarrow}}

\newcommand{\cvgp}{\overset{p}{\to}}

\newcommand{\ga}{\gamma}
\newcommand{\Ga}{\Gamma}
\newcommand{\ep}{\epsilon}
\newcommand{\om}{\omega}

\newcommand{\de}{\delta}

\newcommand{\De}{\Delta}
\newcommand{\sig}{\sigma}

\newcommand{\eps}{\epsilon}

\newcommand{\al}{\alpha}
\newcommand{\Om}{\Omega}

\def \epsilon{\varepsilon}

\def \epsilon{\varepsilon}

\def\l{\ell}

\newcommand{\Var}{\hbox{Var}}

\newcommand{\holder}{H\"{o}lder }

\newcommand{\Rd}{\mathbb{R}^4_\uparrow}

\endlocaldefs

\begin{document}

\begin{frontmatter}
\title{Three-halves variation of geodesics in the directed landscape}
\runtitle{Three-halves variation}

\begin{aug}
\author[B]{\fnms{Duncan} 
	\snm{Dauvergne}
	\ead[label=e1]{duncan.dauvergne@utoronto.ca}},
\author[A]{\fnms{Sourav} \snm{Sarkar}\ead[label=e2,mark]{ss2871@cam.ac.uk}}
\and
\author[B]{\fnms{B\'alint} \snm{Vir\'ag}\ead[label=e3,mark]{balint@math.toronto.edu}
\thanks{B.V. was supported by the Canada Research Chair program and an NSERC Discovery Accelerator grant}}

\runauthor{Dauvergne, Sarkar, and Vir\'ag}
\address[A]{Department of Mathematics,
	University of Cambridge, \printead{e2}}

\address[B]{Department of Mathematics,
	University of Toronto, \printead{e1,e3}}
\end{aug}

\begin{abstract}
We show that geodesics in the directed landscape have $3/2$-variation and that weight functions along the geodesics have cubic variation. 

We show that the geodesic and its landscape environment around an interior point have a small-scale limit. This limit is given in terms of the directed landscape with Brownian-Bessel boundary 
conditions. The environments around different interior points are asymptotically independent. 

We give tail bounds with optimal exponents for geodesic and weight function increments. 

As an application of our results, we show that geodesics are not H\"older-$2/3$ and that weight functions are not H\"older-$1/3$, although these objects are known to be H\"older with all lower exponents.
\end{abstract}

\begin{keyword}[class=MSC2020]
\kwd[Primary ]{60K35}
\kwd[; secondary ]{82B23, 82C22}
\end{keyword}

\begin{keyword}
\kwd{directed landscape}
\kwd{Airy sheet}
\kwd{variation}
\kwd{KPZ}
\kwd{last passage percolation}
\kwd{directed geodesic}
\kwd{geodesic}
\kwd{scaling limit}
\kwd{Brownian-Bessel boundary}
\end{keyword}

\end{frontmatter}


\begin{center}
	\includegraphics[width=5.5in]{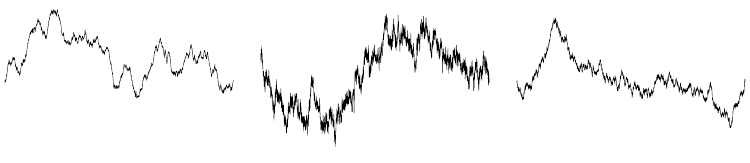}
	
	Which is which? A directed geodesic, its weight function, and a Brownian bridge.
\end{center}

\section{Introduction}\label{S:related}
Kardar, Parisi and Zhang \cite{KPZ86} predicted universal scaling behavior for many planar random growth processes.

The directed landscape $\mathcal L$ was constructed by Dauvergne, Ortmann and Vir\'ag \citep{DOV}, \citDOV\ in the sequel, as the scaling limit of one of these processes, Brownian last passage percolation. 
More precisely, $\mathcal L$ is the four-parameter scaling limit of the Brownian last passage value from line $m$, location $x$ to line $n$, location $y$, see Definition \ref{d:DL}. It is conjectured to be the full scaling limit of all KPZ models. 
The directed landscape $\mathcal L$ is a random continuous function from $$
\R^4_\uparrow = \{(p; q) = (x, s; y, t) \in \R^4 : s < t\}
$$
to $\R$. The value $\cL(p; q)$ is best thought of as a distance between two points $p$ and $q$ in the space-time plane. Here $x, y$ are spatial coordinates and $s, t$ are time coordinates. Unlike with an ordinary metric, $\cL$ is not symmetric, does not assign distances to every pair of points in the plane, and may take negative values. It satisfies the triangle inequality backwards:
\begin{equation}
\label{E:triangle}
\cL(p; r) \ge \cL(p; q) + \cL(q; r) \qquad \mbox{ for all }(p; r), (p; q), (q; r) \in \Rd.
\end{equation}
Just as in true metric spaces, we can define path lengths in $\cL$, see \citDOV, Section 12. First, for a continuous function $\pi:[s, t]\mapsto \R$, referred to as a path, we define its \textbf{length} by
\[\l(\pi)=\inf_{k\in \N}\inf_{s=t_0<t_1<\ldots<t_k=t}\sum_{i=1}^k\cL(\pi(t_{i-1}),t_{i-1};\pi(t_i),t_i)\,.\]
We say that $\pi$ is a {\bf directed geodesic}, or geodesic for brevity, from $(\pi(s),s)$ to $(\pi(t),t)$ if
$\l(\pi)=\cL(\pi(s),s;\pi(t),t)\,.$
By \citDOV, Theorem $12.1$, for any $(p; q)\in \R^4_\uparrow$, almost surely there exists a unique  directed geodesic $\pi$ from $p$ to $q$. Directed geodesics are scaling limits of Brownian last passage paths in the uniform-on-compact topology, see \citDOV, Theorem $1.8$.

Just like one studies the quadratic variation of Brownian motion, it is natural to study the variations of directed geodesics. In this paper, we show that geodesics have $3/2$-variation. \sourav{Brownian motions paths are \holder $1/2^-$ and have finite quadratic variation (defined in a particular way in terms of convergence in probability). On the other hand, directed geodesics are known to be \holder $1/3^-$ (Proposition $12.3$ of \citDOV), and hence can be expected to have finite $3/2$-variation. } Moreover, with the convention $\cL(p,p)=0$ for $p\in \mathbb R^2$, the weight function 
$$
W_\pi(r) := \cL(\pi(s), s; \pi(r), r), \qquad r \in [s, t],
$$
along a geodesic $\pi:[s, t] \to \R$ has a cubic variation. 

To prove these results, we first find the small-scale limit of $\cL$ around a geodesic. Let $(p, q) \in \Rd$, let $\pi$ be a geodesic from $p$ to $q$, and let $r$ be strictly between the time coordinates of $p,q$. For $\e>0$, we consider the \textbf{environment around $(\pi, r)$ at scale $\e$}. This is a quintuple  
$ 
(F_\e, G_\e, \cL_\e, \pi_\e, W_\e)
$
defined as follows. First let $X=X_\e$ be  where 
$
x \mapsto \cL(p; x, r) + \cL(x, r + \e^3; q)
$
is maximized. 
The point $X$ is the approximate location of $\pi$ in the time window $[r, r+\e^3]$. Then define 
\begin{align*}
F_\epsilon(z)&=\Big(\cL(p; X + \e^2 z, r)-\cL(p; X, r)\Big)/\e, \qquad &&z \in \R,\\
G_\e(z) &= \Big(\cL(X + \e^2 z, r + \e^3; q)-\cL(X, r + \e^3; q)\Big)/\e, \qquad &&z \in \R,
\end{align*}
the rescaled initial and final conditions in the window near $(X,r)$, 
\begin{align*}
\cL_\epsilon(z,s;y,t) = \mathcal L(X+ \epsilon^2 z, r + \epsilon^3 s; X+ \epsilon^2 y, r + \epsilon^3 t)/\e, \qquad  z,y\in \mathbb R, \;0\le s< t \le 1,  
\end{align*}
the rescaled landscape in this window, and  
\begin{align*}
\pi_\e(s) &= \Big(\pi(r + s \e^3) - X\Big)/\e^2, \quad &&s \in [0, 1],
\\
W_\e(s) &= \cL(\pi(r), r; \pi(r + s \e^3), r + s\e^3)/\e, \quad &&s \in [0, 1],
\end{align*}
the rescaled geodesic and its weight function in this window. 

For functions $f, g:\R \to \R$, we say that a path $\ga:[s, t] \to \R$ is a \textbf{geodesic from $(f, s)$ to $(g, t)$} if $\ga$ maximizes
$$
f(\ga(s)) + \ell(\ga) + g(\ga(t))
$$
among all paths.

\begin{theorem}[Brownian-Bessel decomposition]\label{t:defgamma}  Let $\pi$ be the unique geodesic from $(0,0) \to (0,1)$.  As $\eps\to 0$, let $t_\e\in (0,1)$ satisfy 
	$
	\e^3/\min(t_\e, 1 - t_\e)\to  0.
	$
	Then the environments around $(\pi,t_\eps)$ at scale $\eps$ satisfy
	$$
	(F_\eps,G_\eps,\mathcal L_\ep, \pi_\eps, W_\e) \cvgd (B- R, -B - R, \cL, \Ga, W_\Ga)
	$$
	as $\e \to 0$,
	where $B$ is a two-sided Brownian motion, $R$ is a two-sided Bessel-$3$ process, $\cL$ is the part of a directed landscape from time $0$ to $1$, $\Ga$ is the almost surely unique $\cL$-geodesic from $(B-R, 0)$ to $(-B - R, 1)$, and $W_\Ga$ is its weight function. The underlying topology is uniform convergence on compact sets, and $B, R,$ and $\cL$ are independent.
\end{theorem}

\sourav{The existence and uniqueness of $\Gamma$ are proved in Corollary \ref{c:Khh-sat}.}
Theorem \ref{t:defgamma} is restated and proven as Theorem \ref{t:vimp'}. 

Theorem \ref{t:defgamma} gives the law of the environment around an interior point of a geodesic for a time interval of length 1. By $1$-$2$-$3$ scaling and  Kolmogorov's extension theorem, this specifies a unique law between times $\pm \infty$, which can be thought of as the directed landscape conditioned to have a bi-infinite geodesic passing through $(0,0)$ \sourav{(also see open problem \ref{op:5})}. If the limit condition on $t_\e$  fails, then we no longer capture a segment of the bi-infinte geodesic and will instead capture a segment of a semi-infinite geodesic starting at $(0,0)$. We do not explore this limit here.

In  Theorem \ref{T:joint-cvg} we show that environments at macroscopically separated times are asymptotically independent. In Theorem \ref{t:protagonist} we show that for some $a>1$,
\[
\E a^{|\Ga(1)-\Ga(0)|^3}<\infty, \qquad  \E a^{|\ell(\Ga)|^{3/2}}<\infty\,.
\]
In fact, such bounds hold uniformly in the prelimit, as we show in Section \ref{S:prelimit-bds}.

\begin{theorem}[Uniform tail bounds for geodesic increments]\label{t:geod-thm}
	Let $\pi$ be the directed geodesic from $(0,0)$ to $(0,1)$. There exists $a>1$, $c > 0$ so that for all $0\le s < s + \e^3=t \le 1$ we have
	\begin{equation*}
	\E a^{|\pi(s) - \pi(t)|^3/\ep^6} < c,  \qquad \E a^{|\cL(\pi(s), s; \pi(t), t)/\ep|^{3/2}} < c.
	\end{equation*}
\end{theorem}

Asymptotic independence of increments and tail bounds allows us to move from convergence of the environment around a geodesic to convergence of variations.

For $f:[s,s+t]\to\mathbb R$ and $\al > 0$, define the $\alpha$-variation of $f$ at scale $\epsilon$ by
$$
V_{\alpha,\epsilon}(f):=\sum_{u \in [s+ \epsilon,t]\,\cap \,\epsilon \mathbb Z}|f(u)-f(u-\epsilon)|^\alpha\,.
$$
For this next theorem, we say that a geodesic $\pi$ from $(x, s)$ to $(y, t)$ is \textbf{leftmost} if $\pi(r) \le \ga(r)$ for all $r \in [s, t]$ for any other geodesic $\ga$ from $(x, s)$ to $(y, t)$. Leftmost geodesics exist between any pair of points  $(p, q) \in \Rd$, see Lemma \ref{L:regularity}. The following theorem is proven in Section \ref{S:variation}.

\begin{theorem}[Variations of geodesics and weights] 
	\label{t:varintro}
	There exists a sequence $\epsilon_n\to 0$ so that with probability 1 the following holds. For all $u=(x,s';y,t')\in \Rd$ and any $[s,t]\subset(s',t')$, the leftmost geodesic $\pi_u$ from $(x,s')$ to $(y,t')$ and its weight function $W_u$ satisfy
	$$
	V_{3/2, \epsilon_n}(\pi_{u}|_{[s,t]})\to (t-s)\nu,\,\qquad V_{3, \epsilon_n}(W_{u}|_{[s,t]})\to (t-s)\mu
	$$
	as $n\to\infty$. Here $\nu=\E |\Ga(1)-\Ga(0)|^{3/2}$ and $\mu=\E|\l(\Gamma)|^3$, where $\Ga$ is as in Theorem \ref{t:defgamma}. 
\end{theorem}

The same statements hold for rightmost geodesics. The caveat that $[s, t]$ cannot be the entire interval $[s', t']$ in Theorem \ref{t:varintro} is due to the fact that we are looking at uncountably many geodesics. If we restrict our attention to a single geodesic, we can remove this restriction. For example, if $\pi$ denotes the almost surely unique geodesic from $(0,0)$ to $(0,1)$ and $W$ is its weight function, then
$$
V_{3/2, \epsilon}(\pi)\cvgp \nu \qquad \text{and} \qquad V_{3, \epsilon}(W)\cvgp \mu\,,
$$
in probability as $\epsilon\to 0$. This is shown in Lemmas \ref{l:firstlemma} and \ref{l:weight}. \sourav{Since the convergence in Theorem \ref{t:varintro} is stated to hold with probability $1$ instead of usual in-probability convergence, we need to resort to a subsequence $\e_n$. }

As discussed above, to prove Theorem \ref{t:varintro} we show that the increments of the directed geodesic are asymptotically independent. In Section \ref{s:holder-not!} we use this to conclude that directed geodesics and weight functions are almost surely not H\"older-$2/3$ and H\"older-$1/3$ respectively. These results complement results from \citDOV\ (see \citep{hammond2020modulus} for analogous results in Poissonian last passage percolation)  showing that directed geodesics are almost surely H\"older-$2/3^-$, and a modulus of continuity bound on $\cL$ that implies weight functions are almost surely  H\"older-$1/3^-$.

\medskip 

\noindent{\bf Related work.}
While the full construction of the directed landscape is quite recent, several aspects of it and of related KPZ models have been studied previously. For a gentle introduction suitable for a newcomer to the area, see Romik \citep{romik2015surprising}. Review articles and books focusing on more recent developments include Corwin \citep{corwin2016kardar}; Ferrari and Spohn \citep{ferrari2010random}; Quastel \citep{quastel2011introduction}; Weiss, Ferrari, and Spohn \citep{weiss2017reflected}; and Zygouras \citep{zygouras2018some}.

The Baik-Deift-Johansson theorem \citep{baik1999distribution} on the length of the longest increasing subsequence was the first to  identify  the single-point distribution of $\cL$ as GUE Tracy-Widom, see also Johansson \citep{johansson2000shape}. Pr\"ahofer and Spohn \citep{prahofer2002scale} proved convergence of last passage values to $\mathcal L(0, 0; y, 1)$ jointly for different values of $y\in \R$. 
Johansson and Rahman \citep{johansson2019multi} and Liu \citep{liu2019multi} independently found formulas for the joint distribution of $\{\cL(p; q_i) : i \in \{1, \dots, k\}\}$, building on work of Johansson \citep{johansson2017two, johansson2018two} and Baik and Liu \citep{baik2017multi}. Matetski, Quastel, and Remenik \citep{matetski2016kpz} derived a formula for the distribution of $h_t(y) = \max_{x \in \R} (g(x) + \cL(x,0; y, t))$ for a fixed function $g$. The function $h_t$ is known as the KPZ fixed point. It is a Markov process in $t$ and is discussed more in Section \ref{s:prelim}.

The works discussed above provide a strong integrable framework for understanding $\cL$ and other KPZ models. More recently, probabilistic and geometric methods have been used in conjunction with a few key integrable inputs to prove regularity results, convergence statements, and exponent estimates in such models.

As an early example of this, Johansson \citep{johansson2000transversal} studied transversal fluctuations. Corwin and Hammond \citep{CH} showed that $\cL_1(y) := \mathcal{L}(0,0; y, 1)$ is the top curve in an ensemble of nonintersecting curves satisfying a certain Brownian resampling property, thereby showing that $\cL_1$ is locally Brownian. 

Hammond \citep{hammond2016brownian, hammond2017modulus, hammond2017patchwork}, Dauvergne and Vir\'ag \citep{DV}, Calvert, Hammond, and Hegde \citep{CHH20}, and Sarkar and Vir\'ag \citep{SV20} improved the understanding of the Brownian nature of $\cL_1$ and related KPZ fixed point height profiles. We use results from \citep{DV}, \citep{CHH20}, and \citep{SV20} as key inputs in our work.

The structure of geodesics both in $\cL$ and other last passage models have been studied previously. Exceptional geodesics in $\cL$ have been studied by Bates, Hammond, and Ganguly \citep{bates2019hausdorff}, building on work of Hammond \citep{hammond2017exponents} and Basu, Ganguly, and Hammond \citep{basu2019fractal}. We use many of the properties established in \citep{bates2019hausdorff} to help analyze geodesic coalescence in Section \ref{S:geod-basic}.
Geodesic properties in prelimiting models have also been studied extensively, see, for example, Basu, Sidoravicius, and Sly \citep{basu2014last}; Basu, Hoffman, and Sly~\citep{basu2018nonexistence}; Basu, Sarkar, and Sly~\citep{basu2017coal}; Georgiou, Rassoul-Agha, and Sepp\"al\"ainen~\citep{georgiou2017stationary}; Hammond and Sarkar~\citep{hammond2020modulus}; and Pimentel~\citep{pimentel2016duality}.


Recent work \citep{SZ20} of Martin, Sly and Zhang establishes the limiting empirical distribution of the lattice environment around a geodesic in exponential last passage percolation. 
The tools for understanding geodesic environments in their discrete and our continuous setting are quite different, and neither result implies the other.  Bates \citep{bates2020empirical} presents an approach to the empirical distributions of geodesic weights which requires no integrable structure and applies to first passage percolation as well. 

\section{Preliminaries}
\label{s:prelim} In this section we recall the relevant results and background material that we will need
in this paper. We first introduce the directed landscape $\cL$ more formally and state a few basic properties. 
\begin{definition} \label{d:DL}
Given $n$, let $B_j,\,j\in \mathbb Z$ be independent copies of two-sided Brownian motion with drift $-2n^{1/3}$ and diffusion parameter 2. Let $a_n=n^{-2/3}/2,b_n=-2n^{-1/3}$, and for $(x,s;y,t)\in \Rd$ define 
	$$
	\cL_n(x,s;y,t)=\sup_{\pi} \sum_{j={\lfloor ns\rfloor}+1}^{{\lfloor nt\rfloor}} \Big(b_n+ B_j(\pi_j)-B_j(\pi_{j-1}-a_n)\Big),
	$$
	where $\pi_{\lfloor s/n\rfloor}=x$, $\pi_{\lfloor t/n\rfloor}=y$, and the $\sup$ is over all finite sequences $\pi_{\lfloor s/n\rfloor+1},\ldots ,\pi_{\lfloor t/n\rfloor-1}$ satisfying  $\pi_{j-1}-a_n\le \pi_{j}$. The {\bf directed landscape} $\cL$ is the distributional limit of $\cL_n$ with respect to uniform convergence on compact subsets of $\Rd$. 
\end{definition}

\begin{figure}
    \centering
    \includegraphics{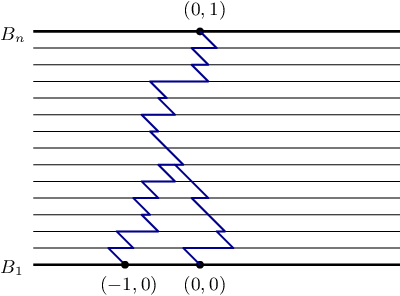}
    \caption{Rescaled geodesics in the prelimiting Brownian model in Definition \ref{d:DL}. Each line is a Brownian motion with negative drift.}
    \label{fig:BLPP}
\end{figure}

See Figure \ref{fig:BLPP} for an illustration of Definition \ref{d:DL}. 
The limit exists by \citDOV,\, Theorem 1.5, and is a random continuous function $\cL:\Rd \to \R$. More precisely, $\cL$ is a random element of the space of continuous functions $\mathcal C(\Rd, \R)$ from $\Rd$ to $\R$ with the Borel $\sigma$-field based on the topology of uniform convergence on compact sets. The marginal $\S(x,y):=\cL(x,0;y,1)$ is known as the \textbf{Airy sheet}, which, in turn, can be used to build $\cL$. We have the following uniqueness theorem, see Definition 10.1 and Theorem 10.9 of \citDOV.

\begin{theorem}
	\label{t:L-unique}
	The directed landscape $\cL:\Rd \to \R$ is the unique, in law, random continuous function satisfying
	\begin{enumerate}[(i)]
		\item (Airy sheet marginals) For any $t\in \R$ and $s>0$ we have 
		$$
		\mathcal{\cL}(x, t; y,t+s^3) \eqd s \S(x/s^2, y/s^2) 
		$$
		jointly in all $x, y$. That is, the increment over the time interval $[t,t+s^3)$ is an \textbf{Airy sheet of scale $s$}.
		\item (Independent increments) For any disjoint time intervals $\{(t_i, s_i) : i \in \{1, \dots k\}\}$, the random functions
		$
		\cL(\cdot, t_i ; \cdot, s_i), i \in \{1, \dots, k \}
		$
		are independent.
		\item (Metric composition law) Almost surely, for any $r<s<t$ and $x, y \in \R$ we have that
		$$
		\cL(x,r;y,t)=\max_{z \in \mathbb R} [\cL(x,r;z,s)+\cL(z,s;y,t)].
		$$
	\end{enumerate}
\end{theorem} 

The metric composition for $\cL$ implies the reverse triangle inequality \eqref{E:triangle}. Note that the reverse triangle inequality is an equality at all points along a geodesic. If $\pi:[s, t] \to \R$ is a geodesic, then for any partition $s = r_0 < r_1 < \dots < r_k = t$ we have
$$
\cL(\pi(s), s; \pi(t), t) = \sum_{i=1}^k \cL(\pi(r_{i-1}, r_{i-1}; \pi(r_i), r_i)).
$$
The directed landscape has invariance properties which we use throughout the paper. 

\begin{lemma} [Lemma 10.2, \citDOV]
	\label{l:invariance} We have the following equalities in distribution as random functions in $\mathcal C(\R^4_\uparrow, \R)$. Here $r, c \in \R$, and $q > 0$.
	\begin{itemize}
		\item[1.] Time stationarity.
		$$
		\displaystyle
		\cL(x, t ; y, t + s) \eqd \cL(x, t + r ; y, t + s + r).
		$$
		\item[2.] Spatial stationarity.
		$$
		\cL(x, t ; y, t + s) \eqd \cL(x + c, t; y + c, t + s).
		$$
		\item[3.] Flip symmetry.
		$$
		\cL(x, t ; y, t + s) \eqd \cL(-y, -s-t; -x, -t).
		$$
		\item[4.] Skew stationarity.
		$$
		\cL(x, t ; y, t + s) \eqd \cL(x + ct, t; y + ct + sc, t + s) + s^{-1}[(x - y)^2 - (x - y - sc)^2].
		$$
		\item[5.] $1:2:3$ rescaling.
		$$
		\cL(x, t ; y, t + s) \eqd  q \cL(q^{-2} x, q^{-3}t; q^{-2} y, q^{-3}(t + s)).
		$$
	\end{itemize}
\end{lemma}

We will also need a series of regularity results about $\cL$. \sourav{The following proposition states that the stationary version of the directed landscape after adding the required parabola has regularity of \holder $1/2^-$ in space and \holder $1/3^-$ in time.} 

\begin{proposition}[Proposition 10.5, \citDOV]
	\label{P:mod-land-i}
	Let $\mathcal R(x,t;y,t+s)=\cL(x,t;y,t+s) + (x-y)^2/s$ denote the stationary version of the directed landscape, and let $K\subset \R^4_\uparrow$ be a compact set. Then for all $u = (x, s; y, t), v = (x', s'; y' ,t') \in K$ we have
	$$
	|\mathcal R(u) - \mathcal R(v)| \le C \lf(\tau^{1/3}\log^{2/3}(\tau^{-1}+1) + \xi^{1/2}\log^{1/2}(\xi^{-1}+1) \rg),
	$$
	where
	$$
	\xi = ||(x, y) - (x', y')|| \qquad \text{ and } \qquad \tau = ||(s,t) - (s', t')||.
	$$
	Here
	$C$ is a random constant depending on $K$ with $\E a^{C^{3/2}} < \infty$  for some $a > 1$.
\end{proposition}

\begin{proposition}[Corollary 10.7, \citDOV]
	\label{P:infty-blowup}
	There exists a random constant $C$ satisfying
	$$
	\P(C > m) \le ce^{-dm^{3/2}}
	$$
	for universal constants $c,d > 0$ and all $m > 0$, such that for all $u = (x, t; y, t + s) \in \R^4_\uparrow$, we have
	\begin{equation}
	\label{E:tail-1}
	\left|\cL(x, t; y, t + s) + \frac{(x - y)^2}{s} \right|\le C s^{1/3} \log^{4/3}\lf(\frac{2(||u|| + 2)^{3/2}}{s}\rg)\log^{2/3}(||u|| + 2).
	\end{equation}
\end{proposition}

\begin{proposition}[Proposition 12.3, \citDOV]
	\label{P:mod-dg}
	There exists a random constant $C$ such that the (almost surely unique) directed geodesic $\pi$ from $(0,0)$ to $(0,1)$ satisfies
	$$
	|\pi(t_1) - \pi(t_2)| \le C|t_1 - t_2|^{2/3}\log^{1/3} \left(\frac{2}{|t_1 - t_2|} \right)
	$$
	for all $t_1, t_2 \in [0, 1]$. Moreover, $\E a^{C^3} < \infty$ for some $a > 1$.
\end{proposition}

To prove coalescence and tightness statements about geodesics we will use almost sure statements about the geodesic structure in $\cL$.

\begin{lemma}
	\label{L:regularity}
	There exists a set $\Om$ of probability one where the following hold for $\cL$:
	\begin{enumerate}[(i)]
		\item For every rational $(p; q) \in \Rd \cap \Q^4$, there is a unique $\cL$-geodesic from $u$ to $v$.
		\item The bound \eqref{E:tail-1} holds for a finite constant $C$.
		\item For every $(p;q ) = (x, s; y, t)\in \Rd$, there exist rightmost and leftmost $\cL$-geodesics $\ga_R$ and $\ga_L$ from $p$ to $q$. That is, there exist geodesics $\ga_R, \ga_L$ from $p$ to $q$ such that for any other geodesic $\ga'$ from $p$ to $q$, we have $\ga_L(r) \le \ga'(r) \le \ga_R(r)$ for all $r \in [s, t]$.
		\item For every compact set $K \sset \Rd$, there exists an $\ep> 0$ such that if $\ga_1, \ga_2$ are distinct $\cL$-geodesics with endpoints $(x, s; y, t), (x', s; y', t) \in K$ such that $||\ga_1 - \ga_2||_u < \ep$, then $\ga_1$ and $\ga_2$ are not disjoint, i.e. $\ga_1(r) = \ga_2(r)$ for some $r \in [s, t]$. Here $||\cdot||_u$ denotes the uniform norm.
		\item For any real numbers $s < t, x_1 < x_2, y_1 < y_2$, and any geodesics $\ga_i$ from $(x_i, s)$ to $(y_i, t)$, $i = 1, 2$, we have $\ga_1(r) \le \ga_2(r)$ for all $r \in [s, t]$.
	\end{enumerate}
\end{lemma}

\begin{proof}
	The fact that the five conditions above are all almost sure follows from Theorem 12.1 in \citDOV\ for (i), Proposition \ref{P:infty-blowup} for (ii), Lemma 13.2 in \citDOV\ for (iii), Theorem 1.18 in \citep{bates2019hausdorff} for (iv), and Lemma 2.7 in \citep{bates2019hausdorff} for (v).
\end{proof}

Much of our analysis in this paper is based around an understanding of the Brownian nature of the \textbf{parabolic Airy process} $\cL_1(y) := \cL(0,0; y, 1)$ and the related KPZ fixed point. The term parabolic comes from the fact that $\cL_1(y) + y^2$ is stationary.

Let $\C[a,b]$ denote the space of all continuous functions $f:[a,b]\mapsto \R$ and $\C_0[a,b]$ denote the space of all continuous functions $f:[a,b]\mapsto \R$ with $f(a)=0$, with the topologies of uniform convergence. Let $\C$ be the space of continuous functions from $\R \to \R$ with the topology of uniform convergence on compact sets.
Fix a diffusion parameter $\sigma^2$; in this paper $\sigma^2=2$. We call a random function $F$ in $\C$ {\bf Brownian on compacts} if for all $y_1<y_2$, the law of the $\C_0[y_1,y_2]$-random variable
$$y\mapsto F(y)-F(y_1)\,,$$
is absolutely continuous with respect to the law of a Brownian motion $B$ on $[y_1, y_2]$ with $B(y_1) = 0$ and diffusion parameter $\sigma^2$.

As discussed in Section \ref{S:related}, Corwin and Hammond \citep{CH} showed that $\cL_1$ is Brownian on compacts. For our purposes, we need something more quantitave than this. The following strong comparison between the parabolic Airy process and Brownian motion is an immediate consequence of Theorem 1.1 in Calvert, Hammond and Hegde \citep{CHH20}.

\begin{theorem}\label{t:airy2}
	The Radon-Nikodym derivative of the law of $\cL_1(\cdot)-\cL_1(y_1)$ on any compact interval $[y_1,y_2]$, with respect to $\mu$, the law
	of Brownian motion with diffusion parameter $2$ on the same interval, is in $L^p(\mu)$ for every $p\in (1,\infty)$.
\end{theorem}
Recall also from Section \ref{S:related} that the KPZ fixed point with an initial condition $h_0:\mathbb R \to \mathbb R\cup \{-\infty\}$ is the process
\[h_t(y) = \sup_{x\in \R} (h_0(x) + \cL(x, 0; y, t))\,,\]
for all $y\in \R$.
Sarkar and Vir\'ag \citep{SV20} showed that the KPZ fixed point is Brownian on compacts. To state their result we need the following definition.
\begin{definition}\label{d:fspace} A function $f:\R\to \R\cup \{-\infty\}$ will be called a $t$-{\bf finitary initial condition} if $f(x)\neq -\infty$ for some $x$, $f$ is bounded from above on any compact interval and
	\[\frac{f(x)-x^2/t}{|x|}\to -\infty\]
	as $|x|\to \infty$.
\end{definition}
The name comes from the fact that $h_t(y)$ is finite for all $x\in \mathbb R$ if and only if the initial condition is $t$-finitary. (Proposition $6.1$ of \citep{SV20}).

\begin{theorem}[Theorem $1.2$, \citep{SV20}] \label{t:gen} Let $t>0$; then for any $t$-finitary initial condition $h_0$ the random function
	\[h_t(y) = \sup_{x\in \R} (h_0(x) + \cL(x, 0; y, t))\,,\]
	is Brownian on compacts.
\end{theorem}
In particular, for $h_0$ bounded above and $h_0\not\equiv -\infty$, $h_t$ is Brownian on compacts for all $t>0$. 

Theorem \ref{t:airy2} also applies to one-dimensional spatial marginals $\cL( \cdot, s; x, t)$ and $\cL(x, s; \cdot, t)$ for arbitrary $x, s, t$ by the invariance properties in Lemma \ref{l:invariance}. Similarly, Theorem \ref{t:gen} applies to KPZ fixed point-like processes started at different times, or run backwards in time.

We can identify locations on geodesics in terms of a maximization problem involving two independent parabolic Airy processes. To analyze this problem, it will be useful to have tail bounds on the value of $\cL_1$ at a point, two-point differences for $\cL_1$, and the location of this maximization.

\begin{theorem}
	\label{T:TW-airy}
	For every $x \in \R$, the random variable $\cL_1(x) + x^2$ has GUE Tracy-Widom distribution. In particular, it satisfies the tail bounds
	\begin{equation*}
	\begin{split}
	c'e^{-d'm^{3/2}} &\le \P(\cL_1(x) + x^2 > m) \le ce^{-dm^{3/2}} \qquad \text{and} \\	c'e^{-d'm^{3}} &\le \P(\cL_1(x) + x^2 < -m) \le ce^{-dm^{3}}
	\end{split}
	\end{equation*}
	for universal constants $c, c', d, d' > 0$.
\end{theorem}
The appearance of the Tracy-Widom random variable in the context of last passage percolation goes back to \cite{baik1999distribution} and the tail bounds above go back to \cite{tracy1994level}. See Exercise 3.8.3 in \cite{anderson2010introduction} for a gentle guide to the tail bound derivation from the asymptotics of Airy functions.

\begin{lemma} [Lemma 6.1, \citep{DV}]
	\label{L:airy-tails}
	There are constants $c, d > 0$ such that for every $x \in \R, y > 0,$ and $a > 0$, we have
	$$
	\P(|\cL_1(x) + x^2 - \cL_1(x + y) - (x+y)^2| > a\sqrt{y}) \le ce^{-da^2}.
	$$
\end{lemma}

Note that in \citep{DV}, the above lemma is stated with the additional condition that $a > 7y^{3/2}$. We can easily remove this condition by applying the one-point bound in Theorem \ref{T:TW-airy}. \sourav{Indeed, by a union bound and that theorem, for all $x, x + y\in \R$ and $0 < a \le 7y^{3/2}$ we have
\begin{eqnarray*}
&&\P(|\cL_1(x) + x^2 - \cL_1(x + y) - (x +y)^2| > a\sqrt{y})\\
 &\le& \P(|\cL_1(x) + x^2|>2^{-1}a\sqrt{y})+\P(|\cL_1(x+y) + (x+y)^2|>2^{-1}a\sqrt{y})\\
 &\le&2ce^{-d2^{-3/2}a^{3/2}y^{3/4}}\le c'e^{-d'a^2}\,,
\end{eqnarray*}
 for universal constants $c', d' > 0$, where in the last inequality we have used that $a \le 7y^{3/2}$.}

For this next lemma, as with Airy sheets we say that $\cL_\sig$ is a \textbf{parabolic Airy process of scale $\sig$} if
$$
\cL_\sig(\cdot) \eqd \sig\cL_1(\cdot/\sig^2),
$$
where $\cL_1$ is a standard parabolic Airy process.

\begin{lemma}[Lemma 9.5, \citDOV]
	\label{L:airy-max}
	Fix $s,t>0$, and let $\cL_s$, $\cL_t$ be two independent parabolic Airy processes of scale $s$ and $t$. Then
	$\cL_s + \cL_t
	$
	has a unique maximum at some location $S$ almost surely.
	Moreover, for any $m > 0$, we have
	$$
	\P(S \not\sset [-ms^2, ms^2]) \le ce^{-dm^3}.
	$$
	Here $c, d>0$ are constants, independent of $s$ and $t$.
\end{lemma}
\sourav{Observe that the above bound does not depend on $t$. Indeed, for the case $s<t$, the parabola $s^2$ is steeper and is the one that determines the overall curvature in the sum $\cL_s+\cL_t$. For $t<s$, $\{S \not\sset [-ms^2, ms^2]\}\subseteq \{S \not\sset [-mt^2, mt^2]\}$, and the latter event is controlled by the previous case.}

\begin{corollary}
	\label{C:geod-locale}
	Let $\pi$ be the (almost surely unique) directed geodesic from $(0,0)$ to $(0, 1)$. Then there are constants $c, d > 0$ such that for any $s \in [0, 1]$, we have
	$$
	\P(s^{-2}|\pi(s^3)| > m) \le c e^{-d m^3}
	$$
\end{corollary}

\begin{proof}
	By the metric composition law for $\cL$, the point $\pi(s^3)$ is the unique maximizer of the function $x \mapsto \cL(0,0; x, s^3) + \cL(x, s^3; 0, 1)$. The processes $\cL(0,0; x, s^3)$ and $\cL(x, s^3; 0, 1)$ are independent Airy processes of scale $s$ and scale $t$, where $t^3=1-s^3$, by the independent increment and invariance properties of $\cL$. The corollary then follows from Lemma \ref{L:airy-max} applied with $t^3=1-s^3$, since $\pi(s^3)$ has the same distribution as $S$ there. 
\end{proof}

Next, recall that the \textbf{two-sided Bessel-$3$ process} $\{R(t):t\in \R\}$ is defined by taking a $3$-dimensional two-sided standard Brownian motion $\{W(t): t\in \R\}$ and defining $R(t)=|W(t)|$.

\begin{theorem}\label{l:bessel} Let $B$ be a standard Brownian motion on $[0,1]$, and let $T$ be the almost surely unique point of maximum of $B$. Let  $\mu$ be the law of
	$$
	B(T)-B(T+t), \qquad t\in [-T, 1 - T]\,.
	$$
	Note that $\mu$ is the law of a random function defined on a random interval. 
	
	Let $R$ be a two-sided Bessel-$3$ process independent of $B$, and let $\nu$ be the law of $R$ on the interval $[-T,1-T]$. 
	
	Then $\mu$ is absolutely continuous with respect to $\nu$ with Radon-Nikodym derivative  $d\mu/d\nu$  in $L^p(\nu)$ for all $p \in (0, 3)$.
\end{theorem}

In Theorem \ref{l:bessel}, we can think of continuous functions on a random closed interval $I \sset [-1, 1]$ as continuous functions on $[-1, 1]$ by extending the functions to be constant off of $I$. This makes the space of such functions a complete separable metric space.

Theorem \ref{l:bessel} follows from fairly standard results about Brownian motions. We include a brief proof for completeness.

\begin{proof}
	Conditioned on $T$, on the interval $I_T :=[- T, 1 - T]$, the process $B(T) - B(T + t)$ is equal in distribution to a two-sided Brownian motion conditioned to remain positive on $I_T\setminus\{0\}$, i.e. a two-sided Brownian meander, see Theorem $1$ of Denisov \citep{denisov}. By Section $4$ of Imhof \citep{Imhof}, the Brownian Radon-Nikodym derivative of this two-sided meander with respect to $\nu$ is given by 
	$$
	(\pi/2)(\sqrt{T}\times {R(T)^{-1}}) (\sqrt{1-T}\times {R(1-T)^{-1}}).
	$$
	Now, by Brownian scaling, the independence of $R$ and $T$, and the independence of the two sides of the Bessel process $R$, the random variables $\sqrt{T}\times {R(T)^{-1}}$ and $\sqrt{1-T}\times {R(1-T)^{-1}}$ are independent, and both are distributed as $|W(1)|^{-1}$, where $W$ is a three-dimensional standard Brownian motion. An elementary calculation yields that $\E |W(1)|^{-p}<\infty$ for all $0 < p < 3$, giving the result.
\end{proof}

Throughout the paper, Brownian on compacts will mean absolute continuity on compacts with respect to Brownian motion of diffusion parameter $2$, unless mentioned otherwise.

\section{Geodesic basics}
\label{S:geod-basic}

In this section we prove coalescence and convergence statements for geodesics in the directed landscape. These will be used throughout the paper. For these lemmas, we will work on the almost sure set $\Om$ introduced in Lemma \ref{L:regularity}.

For the first lemma, for an $\cL$-geodesic $\ga$ from $(x; s)$ to $(y, t)$, we define its {graph}
$$
\mathfrak{g} \ga = \{(\ga(r); r) :r \in [s, t]\}.
$$
Geodesics are continuous by definition, so graphs are closed subsets of $\R^2$.  We consider the convergence of graphs in the Hausdorff metric.

\begin{lemma}
	\label{L:geod-tight}
	Let $(p_n; q_n) \to (p; q) = (x, s; y, t)\in \Rd$, and let $\ga_n$ be any sequence of $\cL$-geodesics from $p_n$ to $q_n$. Then on $\Om$, the sequence of graphs $\mathfrak{g} \ga_n$ is precompact in the Hausdorff metric, and any subsequential limit of $\mathfrak{g} \ga_n$ is the graph of an $\cL$-geodesic from $p$ to $q$.
\end{lemma}

\begin{proof}
	First, for any sequence $m_n \in  \mathfrak{g} \ga_n$, we have
	\begin{equation}
	\label{e:cL-bd}
	\cL(p_n; m_n) + \cL(m_n; q_n) = \cL(p_n; q_n).
	\end{equation}
	Continuity of $\cL$ guarantees that $\cL(p_n; q_n) \to \cL(p; q)$, so the left-hand side of \eqref{e:cL-bd} is bounded below uniformly in $n$. Therefore condition (ii) of Lemma \ref{L:regularity} for the set $\Om$ ensures that all the sets $\mathfrak{g} \ga_n$ must lie in a common compact subset of $\R^2$. In particular, $\mathfrak g \ga_n$ is a precompact sequence in the Hausdorff metric.
	
	Let $\Ga$ be a Hausdorff subsequential limit of $\mathfrak{g} \ga_n$; as such it is a closed subset $\mathbb R \times[s,t]$.  We first show that $\Ga$ cannot have two distinct points with the same time coordinate. Suppose that this is the case. Let $m_1, m_2$ be two such points, and let $m^n_1, m^n_2 \in \mathfrak{g} \ga_n$ be sequences converging to $m_1, m_2$. Without loss of generality, we can assume that the time coordinate of $m^n_1$ is less than the time coordinate of $m^n_2$ infinitely often, and that $m^n_1 \ne p_n, m^n_2 \ne q_n$ for all $n$. Then
	\begin{equation*}
	\cL(p_n; m^n_1) + \cL(m^n_1; m^n_2) + \cL(m^n_2; q_n) = \cL(p_n; q_n)
	\end{equation*}
	for all $n$. The tail bound from condition (ii) on $\Om$ implies that all terms \sourav{other than $\cL(m^n_1; m^n_2)$} are bounded above whereas $\cL(m^n_1; m^n_2) \to -\infty$. This is a contradiction.
	
	We conclude that $\Ga$ is a closed subset of $\mathbb R \times [s,t]$ that  intersects $\mathbb R\times \{r\}$ for every $r\in [s,t]$ in exactly one point.   Thus $\Ga$ is a graph of a continuous function $\pi:[s, t] \to \R$. It remains to check that $\pi$ is a geodesic. Let $m_0 = p, m_1, \dots m_k = q$ be distinct points in $\Ga$, listed in order of increasing time coordinate. We can find sequences $m^n_i \in \mathfrak{g} \ga_n$ converging to $m_i$ for all $i$. Again, we can guarantee that the time coordinates of $m^n_i$ are also listed in increasing order, so
	\begin{equation}
	\label{E:cL-pni}
	\sum_{i=1}^k \cL(m^n_i; m^n_{i-1}) = \cL(p^n_0; p^n_k).
	\end{equation}
	Continuity of $\cL$ on $\Rd$ implies that this equality passes to the limit, showing that $\pi$ is a geodesic.
\end{proof}

\begin{corollary}
	\label{C:precompact}
	Let $K \sset \Rd$ be any compact set. Then on $\Om$, the set of geodesic graphs 
	$$
	H_K = \{\mathfrak{g} \ga : \ga \text{ is a geodesic with endpoints in } K \}
	$$
	is closed in the Hausdorff metric.
\end{corollary}

\begin{proof}
	If not, then we could find a sequence of geodesic graphs $\mathfrak{g} \ga_n \in H_K$ without any limit points. However, by passing to a subsequence, we can ensure that the endpoints of $\ga_n$ converge to some $(p; q) \in K$, and so Lemma \ref{L:geod-tight} implies that the geodesic graphs have subsequential limits in $H_K$.
\end{proof}

Next, we show that convergence of geodesic graphs implies that the geodesics eventually overlap.

\begin{lemma}
	\label{L:overlap} Fix $\om \in \Om$. Let $(p_n; q_n) = (x_n, s_n; y_n, t_n) \in \Rd \to (p; q) = (x, s; y, t)\in \Rd$, and let $\ga_n$ be any sequence of $\cL$-geodesics from $p_n$ to $q_n$. Suppose also that either
	\begin{enumerate}
		\item[(i)] $(p_n; q_n) \in \Q^4$ for all $n$ and $\mathfrak{g}\ga_n \to \mathfrak{g} \ga$ for some $\cL$-geodesic $\ga$ from $p$ to $q$, or
		\item[(ii)] there is a unique geodesic $\ga$ from $p$ to $q$.
	\end{enumerate} Then the  \textbf{overlap}
	$$
	O(\ga_n, \ga) = \{r \in [s_n, t_n] \cap [s, t] : \ga_n(r) = \ga(r) \}
	$$
	is an interval for all $n$ whose endpoints converge to $s$ and $t$.
\end{lemma}

\begin{proof}
	First suppose $r < r' \in O(\ga_n, \ga)$. Since either $\ga_n$ is the unique geodesic from $p_n$ to $q_n$ or $\ga$ is the unique geodesic from $p$ to $q$, there must also be a unique geodesic from $(\ga_n(r), r)$ to $(\ga_n(r'), r')$. Since $\ga|_{[r, r']}$ and $\ga_n|_{[r, r']}$ are two such geodesics, they must be equal, and so $[r, r'] \in O(\ga_n, \ga)$. Hence $O(\ga_n, \ga)$ is an interval.
	
	Now, if we are in case (ii) above, then Lemma \ref{L:geod-tight} and uniqueness of $\ga$ implies that $\mathfrak{g} \ga_n \to \mathfrak{g} \ga$, as in case (i).
	
	Recall that Hausdorff convergence of graphs of continuous functions implies uniform convergence of the functions themselves. Therefore in both cases, the continuity of $\ga_n$ and $\gamma$ imply that for any $[s', t'] \sset (s, t)$, $\ga_n|_{[s', t']} \to \ga|_{[s', t']}$ uniformly. Note here that $\ga_n$ may not be defined on $[s', t']$ for small $n$. This does not affect the uniform convergence statement. Therefore by condition (iv) on $\Om$, for all large enough $n$ there is a time $q_n \in [s', t']$ with $\ga_n(q_n) = \ga(q),$ so $q_n \in O(\ga_n, \ga)$. Therefore $O(\ga_n, \ga)$ eventually intersects any closed subinterval of $(s, t)$, so the endpoints of $O(\ga_n, \ga)$ must converge to $s$ and $t$.
\end{proof}

\begin{corollary}
	\label{c:coal-2}
	Again, we work on $\Om$. Let $\ga$ be the geodesic from $(0,0)$ to $(0,1)$. Let $p = (x, t) \in \mathfrak{g} \ga, p \ne (0,0)$, and let $p_n \to p$. Let $\ga_n$ be any sequence of geodesics from $(0,0)$ to $p_n$. Then
	$O(\ga_n, \ga) = [0, t_n]$ for a sequence $t_n \to t$.
\end{corollary}

\begin{proof}
	Since $\ga$ is unique, $\ga|_{[0, t]}$ is a unique geodesic from $(0,0)$ to $p$. Therefore by Lemma \ref{L:overlap}, $O(\ga_n, \ga)$ is an interval converging to $[0, t]$. The interval contains $0$ since $\ga$ and $\ga_n$ have the same starting point for all $n$.
\end{proof}

As a consequence of Lemma \ref{L:geod-tight} and Lemma \ref{L:overlap}, we can show that all rightmost and leftmost geodesics can be covered by the set of geodesics between rational points.
\begin{corollary}
	\label{C:geodesic-frame}
	On $\Om$ the following holds. Let $(p, q) = (x, s; y, t) \in \Rd$ and let $\ga$ be either the rightmost or leftmost geodesic from $p$ to $q$. Then for any $[a, b] \sset (s, t)$, there is a geodesic $\hat \ga$ between rational points $(\hat x, \hat s)$ and $(\hat y, \hat t)$ with $\hat s \le a$ and $b \le \hat s$ such that $\ga|_{[a, b]} = \hat \ga|_{[a, b]}$.
\end{corollary}

\begin{proof}
Without loss of generality, we assume that $\ga$ is the rightmost geodesic from $p$ to $q$. Let $s' \in (s, a), t' \in (b, t)$ be rational times.
	Let $(p_n; q_n) =  (x_n, s'; y_n, t') \in \Rd \cap \Q^4$ be a sequence converging to $(\ga(s'), s'; \ga(t'), t')$ with $\ga(s')<x_n, \ga(t')<y_n$ for all $s', t'$, and let $\ga_n$ denote the unique geodesic from $p_n$ to $q_n$.
	
	By Lemma \ref{L:geod-tight}, the sequence $\mathfrak{g}\ga_n$ is precompact in the Hausdorff topology with subsequential limits that are graphs of geodesics from $(\ga(s'), s')$ to $(\ga(t'), t')$. Moreover, by the geodesic ordering property (v) of $\Om$, $ \pi\le \gamma_n$ for any geodesic $\pi$ from $(\ga(s'), s')$ to $(\ga(t'), t')$. In particular, this means that the only candidate limit point is the graph of the rightmost geodesic, $\mathfrak{g}\ga|_{[s', t']}$. Lemma \ref{L:overlap} then implies that $\ga = \ga_n$ on $[a, b]$ for large enough $n$, as desired.
\end{proof}

\section{Brownian-Bessel decomposition}
In this section, we prove a Brownian-Bessel decomposition around an interior geodesic point. 
Consider functions $h_0:\R\mapsto \R \cup \{-\infty\}, g_0:\R\to \R \cup \{-\infty\}$. Let $t>1$ and
$$
h_1(x)=\sup_{z\in \R} \left(h_0(z)+ \mathcal L(z,0;x,1)\right) \quad \text{and} \quad g_1(x) = \sup_{y \in \R} (\mathcal L(x,1;y,t) + g_0(y))\,,
$$
for all $x\in \R$. Then we have the following lemma.

\begin{lemma}\label{l:basic} Assume that $h_0, g_0$ are not identically equal to $-\infty$ and satisfy the bound
	$$
	h_0(x) < c - |x|, \qquad g_0(x) < c - |x|
	$$
	for some constant $c$.
	Then the function $h_1+g_1$ attains a maximum at a unique location $X$ almost surely.
\end{lemma}

The decay condition on $h_0, g_0$ above is far from optimal, but will be sufficient for our purposes.
\begin{proof}
	Because of the independent increment property of the directed landscape, $h_1(\cdot)$ and $g_1(\cdot)$ are independent. Also, if $B$ denotes a standard Brownian motion and $r(x)$ is any continuous function, then $B+r$ has a unique point of maximum on any compact set almost surely. The same holds if $B$ is only Brownian  on compacts and $r$ is replaced by a random continuous function independent of $B$. The process $h_1$ is Brownian on compacts by Theorem \ref{t:gen}, so $h_1 + g_1$ has a unique maximum on any compact set almost surely.
	Now, part (ii) of Lemma \ref{L:regularity} and the decay bound on $h_0, g_0$ implies that for a random $N \in \N$, we have
	$$
	\max_{x\in [-N, N]} (h_1(x) + g_1(x)) = \max_{x\in \R} (h_1(x) + g_1(x)) > \max_{x\in [-N, N]^c} (h_1(x) + g_1(x)).
	$$
	Since $\max_{x\in [-N, N]} (h_1(x) + g_1(x))$ is uniquely attained almost surely, the lemma follows.
\end{proof}

The following theorem is the main result of this section.

\begin{theorem}\label{t:brbess} Let $h_0,h_1,g_0, g_1$ be as in Lemma \ref{l:basic}.
	Let $X$ be the unique point of maximum of $h_1(x)+g_1(x)$. Let $\ep \in (0, 1)$ and $f_\epsilon,g_\epsilon$ be the rescaled versions of $h_1$ and $g_1$ near $X$, that is,
	\begin{align*}
	f_\epsilon(z)&=\epsilon^{-1}(h_1(X+\epsilon^2z)-h_1(X)),\\
	g_\epsilon(z)&=\epsilon^{-1}(g_1(X+\epsilon^2z)-g_1(X)).
	\end{align*}
	Then as $\epsilon\to 0$, the pair
	\[(f_\e,g_\e)\cvgd (-R+B,-R-B).\]
	That is, the joint distribution of $(f_\e,g_\e)$ converges in law to that of the pair $(-R+B,-R-B)$ in the uniform-on-compact topology, where $R$ is a standard two-sided Bessel-$3$ process and $B$ is an independent standard two-sided Brownian motion.
	
	Moreover, for any $K > 0$, the law of $(f_\ep,g_\ep)$ on the box
	$$
	[-K\e^{-2}, K\e^{-2}]\times [-K\e^{-2}, K\e^{-2}]
	$$
	is absolutely continuous with respect to $(-R+B,-R-B)$ on that box, and the Radon-Nikodym derivatives are tight for $\e \in (0,1)$.
\end{theorem}

Note that the Brownian motion $B$ and the Bessel process $R$ above have diffusion parameter $1$, so their sum and difference have diffusion parameter $2$.
Before proving this theorem, we will need the following lemmas. 

\begin{lemma}\label{l:locbrown} Let $W(\cdot)$ have law which is Brownian on compacts (with diffusion parameter $2$). Then
	$$
	W_\epsilon(t)=\epsilon^{-1}(W(\epsilon^2 t)-W(0))$$
	converges in law to a Brownian motion $B$ (with diffusion parameter $2$) as $\e\to 0$, in the uniform-on-compact topology.
\end{lemma}
\begin{proof} 
	We prove this lemma by contradiction.  Assume the contrary. Then there exists a measurable set $A$ and $\delta>0$ so that
	\[|\P(W_\epsilon \in A)-\P(B \in A)|>\delta\]
	along some subsequence of $\epsilon$. Without loss of generality, assume that \[\P(W_\epsilon\in A) < \P(B\in A)-\delta\]
	along that subsequence. We can choose the event $A$ so that it depends only on $W_\epsilon(\cdot) - W_\epsilon(s)$ restricted to a time interval $[s,t]$ for some $0<s<t$, since the law of such incremental segments determines the law of any continuous process.
	
	Pick a further subsequence $\epsilon_n$ so that $\epsilon_n/\epsilon_{n+1}>(t/s)^{1/2}$. Let $X_n=1(W_{\epsilon_n} \in A)$, and let $Y_n=(X_1+\cdots+X_n)/n$. Then
	\[\E Y_n=\frac{1}{n}\sum_{i=1}^n\P(W_{\e_i}\in A)< \P(B\in A)-\delta\,,\]
	for all $n$. Hence
	\begin{equation}\label{e:wtf}
	\limsup_n \E Y_n\leq \P(B\in A)-\delta\,.\
	\end{equation}
	Now, if the law of $W$ is Brownian motion, then the $X_i$ are i.i.d, so the law of large numbers implies
	\[Y_n \to \P(B \in A)\quad \quad \mbox{almost surely}\,.\]
	Since this convergence is almost sure, the same holds if the law of $W$ is absolutely continuous with respect to Brownian motion on $[0,1]$. The bounded convergence theorem then implies that
	\[\E Y_n \to \P(B \in A)\,,\]
	which contradicts \eqref{e:wtf}.
\end{proof}

\begin{remark}\label{r:brhighdim} The same result with the same proof holds for two-sided Brownian motion scaled near zero, in arbitrary dimensions, that is, let $W$ have law which is absolutely continuous with respect to $B_d$, a $d$-dimensional Brownian motion, on compact sets for some $d\in \N$. Then
	$$W_\epsilon(t)=\epsilon^{-1}(W(\e^2t)-W(0))$$
	converges in law to $B_d$ as $\e\to 0$ in the uniform-on-compact topology.
\end{remark}

The next lemma relates Theorem \ref{l:bessel} to processes that are Brownian on compacts.
\begin{lemma}\label{l:bessel2} Let $D:\R \to \R$ have law which is Brownian on compacts (with diffusion parameter $2$), and suppose that the maximum of $D$ is almost surely uniquely attained at a time $T$. Then
	$$
	R_\epsilon(t)=\epsilon^{-1}(D(T) - D(T+\epsilon^2 t))\,, \qquad t\in [-1,1]\,,
	$$
	converges in law to $\mu_R$, the law of a two-sided Bessel-$3$ process $R$ (with diffusion parameter $2$) on $[-1,1]$ as $\e\to 0$. Moreover, $R_\ep$ is absolutely continuous with respect to $R$ on $[-\ep^{-2}, \ep^{-2}]$, with tight Radon-Nikodym derivatives for $\e \in (0, 1)$.
\end{lemma}

\begin{proof} It suffices to show that the desired convergence and tightness holds conditionally on $|T| \le m$ for all $m$ large enough so that $\P(|T| < m) > 0$. On this event, by Theorem \ref{l:bessel} and Brownian scaling, $R_1$ is absolutely continuous with respect to a Bessel process restricted to $|t|\leq m$. In particular, we can realize $R_1(t)=|W(t)-W(0)|$ where $W$ is absolutely continuous with respect to three-dimensional Brownian motion. By Remark \ref{r:brhighdim}, $\epsilon^{-1}(W(\epsilon^2t)-W(0))$ converges in law to a three dimensional Brownian motion. Hence $R_\e(t)=\epsilon^{-1}|W(\epsilon^2t)-W(0)|$ converges to a two-sided Bessel-$3$ process by the continuous mapping theorem.
	
	The tightness of Radon-Nikodym derivatives follows from the fact that $R_\e|_{[-\e^2, \e^2]}$ is related to $R_1$ on $[-1, 1]$ by Brownian scaling. 
\end{proof}

Combining these, we have the following lemma. All the Brownian motions and Bessel processes considered here have diffusion parameter $2$.

\begin{lemma}\label{l:2coord}Let $(D,W):\R \to \R^2$ have law which is Brownian on compacts in $\R^2$ (\sourav{that is, on any compact set in $\R^2$, its law is absolutely continuous with respect to that of a Brownian motion in $\R^2$}), and suppose that the maximum of $D$ is almost surely uniquely attained at a time $T$. Then
	\[(D_\e(t),W_\e(t)):=(\e^{-1}(D(T)-D(T+\e^2t)),\e^{-1}(W(T)-W(T+\e^2t)))\,,\qquad t\in [-1,1]\]
	converges in law to $(R,B)$ where $R$ is a two-sided Bessel and $B$ is an independent two-sided Brownian motion.
	
	Moreover, $(D_\e(t),W_\e(t))$ is absolutely continuous with respect to $(R, B)$, and the Radon-Nikodym derivatives on $[-\ep^{-2}, \ep^{-2}]^2$ of $(D_\e(t),W_\e(t))$ with respect to $(R, B)$ are tight for $\e \in (0, 1)$.
\end{lemma}

\begin{proof} Conditioning on the maximum of the first coordinate and using the independence of the coordinates of the two-dimensional Brownian motion, the first part of the lemma follows from Lemmas \ref{l:locbrown} and Lemma \ref{l:bessel2}. The tightness of Radon-Nikodym derivatives follows from the fact that $(D_\e, W_\e)|_{[-\e^2, \e^2]^2}$ is related to $(D_1, W_1)$ on $[-1, 1]^2$ by Brownian scaling.
\end{proof}
Finally we are ready to prove Theorem \ref{t:brbess}.
\begin{proof}[Proof of Theorem \ref{t:brbess}]
	By Theorem \ref{t:gen}, $h_1(\cdot)$ and $g_1(\cdot)$ are Brownian on compacts (with Brownian motion of diffusion parameter $2$), and independent, so $2^{-1/2}(h_1 + g_1, h_1 - g_1)$ is Brownian on compacts as well. The convergence and the Radon-Nikodym derivative claim for $K=1$ then follows from Lemma \ref{l:2coord} \sourav{(together with the fact that $B\overset{d}{=}-B$)}. The Radon-Nikodym derivative claim for general $K >0$ follows from Brownian scaling.	
\end{proof}

\section{Compressed exponential tails for displacement and length}

To move from the Brownian-Bessel decomposition of Theorem \ref{t:brbess} to the stronger Brownian-Bessel-landscape decomposition of Theorem \ref{t:defgamma}, we need to show that the random function
\begin{equation}\label{E:formerH}
B(x)-R(x)+\mathcal {L}(x,0;y,1)-B(y)-R(y)
\end{equation}
takes its maximum at some finite values $X,Y$. Here $B$ is a Brownian motion, $R$ is a Bessel-$3$ process, $\mathcal L$ is a directed landscape from time $0$ to time $1$, and all three objects are independent. We will later show that $Y - X$ and $\cL(X, 0; Y, 1)$ are the limiting spatial increment and length increment along a geodesic. In this section we show compressed exponential tail bounds on these increments. Later in  Lemma \ref{l:uniquegeo} we will show that the maximum is attained at a unique location, but we do not use  this fact in this section. 

\begin{theorem}\label{t:protagonist} 
	The function \eqref{E:formerH}  attains its maximum at some finite value almost surely.  Moreover, there are constants $c>0$, $a>1$ so that for any measurably chosen optimizers $X,Y$ we have 
	\begin{equation}\label{e:boundX}
	P(|X|>t)=P(|Y|>t)<c\,\frac{\log t}{ \sqrt{t}}\qquad \mbox{ for all } t\ge2, 
	\end{equation}
	\[\E a^{|Y-X|^3}<\infty, \qquad  \E a^{|\cL(X,0;Y,1)|^{3/2}}<\infty\,.\]
\end{theorem}
\begin{remark}
	We believe that  $P(|X|>t)>c/\sqrt{t}$, since the Bessel process $R$ has  probability of this order to be less than $1$ after time  $t$. Also, we think the upper bound could be updated to match this, as the total time $R$ spends below $1$ has exponential tails by the Ciesielski-Taylor identity. 
\end{remark}

To prove Theorem \ref{t:protagonist}, we will use a simple tail bound on the Airy sheet $\S$ which we record as a lemma here for later use.

\begin{lemma}
	\label{L:sheet-bd}
	Let $\S(x,y)=\cL(x,0;y,1)$ be the Airy sheet. Then there exist constants $c, d > 0$ such that for all $x, y \in \R$ we have
	\begin{equation}
	|\S(x,y)+(x-y)^2|\le C_1+c\log^{2/3}(2+|x|+|y|),
	\end{equation}
	where $C_1$ satisfies the tail bound
	\begin{equation}\label{e:C1tail}
	\P(C_1>m)\leq ce^{-dm^{3/2}}.
	\end{equation}
\end{lemma}

\begin{proof}
	Proposition \ref{P:infty-blowup} implies that
	\begin{equation}\label{e:sheet-tail}
	\P\Big(\sup_{x,y\in[0,1]} |\S(x,y)+(x-y)^2|>m\Big)\leq ce^{-dm^{3/2}} \qquad \mbox{ for all }m\ge 0. 
	\end{equation}
Now, the process $\S(x, y) + (x-y)^2$ is stationary in $x$ and $y$ by the skew-stationarity property in Lemma \ref{l:invariance}, so by a union bound
\begin{align*}
\P\Big(&\sup_{x,y\in \R} |\S(x,y)+(x-y)^2| - c_2\log^{2/3}(2+|x|+|y|) >m\Big)\\
&\leq \sum_{(n, \ell) \in\Z^2} \P\Big(\sup_{(x,y)\in [n, n+1] \times [\ell, \ell + 1]} |\S(x,y)+(x-y)^2| - (c_2/2)\log^{2/3}(2+n + \ell) >m\Big) \\
&\leq c_1e^{-dm^{3/2}}
\end{align*}
as long as $c_2$ is large enough. Setting $c = c_1 \vee c_2$ then gives the result.
\end{proof}

\begin{proof}[Proof of Theorem  \ref{t:protagonist}]
	Let $\S(x,y)=\cL(x,0;y,1)$. Lemma \ref{L:sheet-bd} implies that for all $r,x\in \mathbb R$,
	\begin{equation}\label{e:eqq2}
	|\S(x,x+r)+r^2|\le C_1+c\log^{2/3}(2+|x|)+c\log^{2/3}(2+|r|)\,,
	\end{equation}
	where $C_1$ satisfies the tail bound in \eqref{e:C1tail}.
	 Now, we have the standard estimate
\begin{equation*}
\P\lf(\sup_{x \in [a, a + 1], r \in [-t, t] }\frac{|B(x+r)-B(x)|}{1+|t|^{1/2}} > m \rg) \le c e^{-m^2/2}
\end{equation*}
on the tails of a standard Brownian motion on $\R$, where $c > 0$ is an absolute constant. Using that $R$ is simply the modulus of a $3$-dimensional Brownian motion, and a union bound as in the proof of Lemma \ref{L:sheet-bd}, we get that
there exists $c, d > 0$ such that for all $x,r\in \R$ we have
	\begin{equation}\label{e:eqq3}
	\frac{|(R+B)(x+r)-(R+B)(x)|}{1+|r|^{1/2}}\le C_2 +c\log^{1/2}(2+|x|)+c\log^{1/2}(2+|r|)\,,
	\end{equation}
	where $C_2$ satisfies the tail bound
	\begin{equation}\label{e:C2tail}
	\P(C_2>m)\leq ce^{-dm^{2}}\,.
	\end{equation}
	In particular, \eqref{e:eqq2}  and \eqref{e:eqq3} hold for the random variables $X$ and $D=Y-X$ in place of $x$ and $r$ respectively. For an optimizer $(X,X+D)$ we have
	\begin{equation}\notag\label{e:eqq1}
	\S(X,X+D)-(R-B)(X)-(R+B)(X+D)\ge \S(0,0)\,.
	\end{equation}
	Combining this with \eqref{e:eqq2} and \eqref{e:eqq3} with $X,D$ in place of $x,r$, we get 
	\begin{equation}\label{e:eqqimp}
	\begin{split}
	D^2 &\le -\S(X,X+D)+     C_1+c\log^{2/3}(2+|X|)+c\log^{2/3}(2+|D|)\\
	&\le -\S(0,0)-2R(X)+ C_1+C_2|D|^{1/2}\\
	&+c\left(\log^{2/3}(2+|X|)+\log^{2/3}(2+|D|)\right)(1+|D|^{1/2})
	\end{split}
	\end{equation}
	where $C_1, C_2$ are new random constants that also satisfy the tail bounds \eqref{e:C1tail}, \eqref{e:C2tail}. 
	
	\medskip \noindent {\bf The bound on $X$.} 
	For any $\delta>0$, by the tail bounds \eqref{e:sheet-tail}, \eqref{e:C1tail}, and \eqref{e:C2tail} the event 
	$$A_t=\Big\{-\S(0,0),C_1,C_2\leq \log^{2/3+\delta}(2+t)\Big\}, \qquad \text{ satisfies } \qquad \P A_t^c\le c e^{-d\log^{1+3\delta/2}(2+t)}.$$ Since $R(X)\ge 0$, on the event $A_t\cap \{|X|>t\}$ the inequality \eqref{e:eqqimp} implies  \[D^2\le  c_\delta|D|^{1/2}\log^{2/3+\delta}(2+|X|)\,,\]
	so that we have $|D|\le  c_\delta\log^{4/9+\delta}(2+|X|)$. Therefore, still on the event  $A_t\cap \{|X|>t\}$, with a choice of $\delta$ so that $2/3+\delta+(4/9+\delta)/2<1$, the inequality   \eqref{e:eqqimp} implies
	$R(X)<c\log(2+|X|)$.
	Thus
	\begin{eqnarray*}
		\P(|X|>t)&\leq&  \P(|X|>t, R(X)<c\log(2+|X|))+\P A_t^c\\
		&\le& \P\left(\min_{x>t}\frac{R(x)}{c\log(2+x)}<1\right)+e^{-d\log^{1+3\delta/2}(2+t)}\,.
	\end{eqnarray*}
	Now for $k\ge 1$ we have
	\begin{eqnarray*}
		\P\left(\min_{x\in[k^3t,(k+1)^3t]}R(x)<c\log(2+(k+1)^3t)\right)&\le& \P\left(\inf_{x>k^3t}R(x)<c\log(2+(k+1)^3t)\right)\\
		&\le& d\log(2+(k+1)^3t)k^{-3/2}t^{-1/2}\,,
	\end{eqnarray*}
	for some universal constant $d>0$. \sourav{We used the Brownian scaling invariance of $R$ and the fact that its future infimum after time $1$ has bounded density. Indeed, by Pitman's $2M-B$ Theorem (Theorem 1.3(iii) of Pitman \cite{Pi75}), the future infimum after time $1$ of $R$ is distributed as the maximum until time $1$ of a standard Brownian motion, which in turn is distributed as $|Z|$ where $Z\sim N(0,1)$.}
	Summing over integer values of $k$ we get \eqref{e:boundX}.
	
	\medskip
	
	\noindent{\bf The bound on $Y-X=D$}. In  \eqref{e:eqqimp} we use that $R(X)>0$, break the rest of the right-hand side into a sum of four terms and bound it by  $4$ times the maximum. We conclude that $D^2/4$ is bounded above by one of 
	$$
	-\S(0,0)+C_1, \qquad  C_2|D|^{1/2}, \qquad c \log^{2/3}(2+|X|),\qquad c\log^{2/3}(2+|D|)(1+|D|^{1/2}).
	$$
	The last one of these would imply  $|D|$ is bounded by some $c$. The tail bounds 
	\eqref{e:sheet-tail}, \eqref{e:C1tail}, \eqref{e:C2tail} and \eqref{e:boundX}
	now imply that $\E a^{|D|^3}<\infty$ for some $a>1$.
	\medskip

	\noindent{\bf The bound on $\cL(X,0;Y,1)=\S(X,X+D)$}. 
	By \eqref{e:eqq2} with $x = X, r = D$, we get 
	$$
	|\S(X,X+D)|\le C_1+c\log^{2/3}(2+|X|)+D^2 + c\log^{2/3}(2+|D|).$$
	Bounding the sum of three terms by three times the maximum, we get that $|S(X,X+D)|/3$ is bounded above by one of 
	$$
	C_1,\qquad c\log^{2/3}(2+|X|), \qquad  D^2+c\log^{2/3}(2+|D|).
	$$
	The tail bounds 
	\eqref{e:C1tail}, \eqref{e:boundX} and $\E a^{|D|^3}<\infty$ now imply that $\E {a'}^{|S(X,X+D)|^{3/2}}<\infty$ for some $a'>1$.
\end{proof}

\begin{remark}
	\label{R:bounded-domain}
	Theorem \ref{t:protagonist} holds with the same constants $c, a$ if the processes $B$ and $R$ are replaced by truncated processes $B^I, R^I$, where $I$ is a closed subinterval of $\R$, and the function $f^I$ is equal to $f$ on $I$ and $-\infty$ on $I^c$. The proof goes through verbatim. This is also the case if $I$ is a random interval, independent of $R, B,$ and $\cL$. We will use this to derive moment bounds in the prelimit in Section \ref{S:prelimit-bds}.
\end{remark}

\section{Environment around a geodesic}
\label{S:environment}
The main goal of this section is to prove Theorem \ref{t:defgamma}. 
First, we show existence of geodesics with initial and final conditions. Recall from the introduction that for functions $h_0, g_0$, that $\gamma$ is a geodesic from $(h_0, 0)$ to $(g_0, s)$ if it is the argmax of
$$
h_0(\gamma(0))+\l(\gamma) + g_0(\gamma(s))
$$
over all continuous paths from time $0$ to $s$, where $\ell(\ga)$ is the length of $\ga$ in $\cL$.

\begin{lemma}\label{l:uniquegeo} Let $h_0, g_0:\R \to \R \cup\{-\infty\}$ be functions that are locally bounded above. Let $t>s\ge 0$, and define
	$$
	H(x, y) = h_0(x) + \cL(x, s; y, t) + g_0(y).
	$$ 
	Suppose that $H$ achieves its maximum almost surely, and that the set of points in $\R^2$ where it is maximized is almost surely bounded.
	Then almost surely, there exists a unique geodesic $\gamma$ from $(h_0,s)$ to $(g_0,t)$.
\end{lemma}

\begin{proof}
	First define $H^b$ in the same way we defined $H$, but with $h_0, g_0 = -\infty$ outside of some compact subset $[-b, b] \sset \R$. By Theorem \ref{t:gen}, the function 
	$$
	h'(x) = \sup_{x \in [-b, b]} \Big(h_0(x) + \cL(x, s; y, t)\Big)
	$$
	is Brownian on compacts, so $h'+g$ almost surely has a unique maximum $X^b$ on $[-b, b]$ for any continuous $g$, and hence on $\R$. Therefore any $\argmax$ of $H^b$ is attained on the set $\{X^b\} \times \R$. A similar argument with the roles of $h_0, g_0$ reversed shows that argmax of $H^b$ must be attained on a set $\R \times \{Y^b\}$, and so the argmax of $H^b$ is almost surely uniquely attained at $(X^b, Y^b)$.
	
	Now, Lemma \ref{L:regularity}(iii) ensures that almost surely there exists at least one geodesic from $(X^b, s)$ to $(Y^b, t)$. Next we check that this geodesic is almost surely unique. 
	If there were two such geodesics, then by continuity, they would differ at a rational time $r$, implying that the function
	$$
	\lf[\max_x \big(h_0(x) + \cL(x, s; z, r)\big)\rg] + \lf[\max_y \big(\cL(z, r; y, t) + g_0(y)\big)\rg]
	$$ 
	would be maximized at two points. This is almost surely not the case by Lemma \ref{l:basic} and rescaling properties of $\cL$, see Lemma \ref{l:invariance}. 
	
	Finally, since the set of points $S$ where $H$ achieves its maximum is almost surely bounded, $H$ is uniquely maximized at a point $(X^N, Y^N)$ for some large random $N\in \N$, and almost surely, the only geodesic from $(h_0,s)$ to $(g_0, t)$ is the unique geodesic from $(X^N,s)$ to $(Y^N, t)$.
\end{proof}

The conditions of Lemma \ref{l:uniquegeo} are satisfied in cases that we care about.
\begin{corollary}
	\label{c:easy}
	If $h_0, g_0$ satisfy the decay conditions
	$$
	h_0(x) < c - |x|, g_0 < c - |x|
	$$
	of Lemma \ref{l:basic}, then the function $H$ in Lemma \ref{l:uniquegeo} has a unique maximizer almost surely, and there is almost surely a unique geodesic $\Ga$ from $(h_0, 0)$ to $(g_0, t)$.
\end{corollary} 

For this next corollary and throughout this section, $B$ will denote a two-sided Brownian motion and $R$ will denote a two-sided Bessel-$3$ process, both with diffusion parameter $1$. These processes are independent of each other and any landscapes $\cL$ that they interact with.

\begin{corollary}
	\label{c:Khh-sat}
	The function
	$$
	H(x, y) = B(x) - R(x) + \cL(x, 0; y, 1) - B(y) - R(y),
	$$
	has a unique maximizer almost surely, and there is almost surely a unique geodesic $\Ga$ from $(B-R, 0)$ to $(-B-R, 1)$.
\end{corollary}

\begin{proof}[Proofs of Corollaries \ref{c:easy}, \ref{c:Khh-sat}]
	The conditions of Lemma \ref{l:uniquegeo} hold by Lemma \ref{l:basic} for the case of Corollary \ref{c:easy}, and hold almost surely by Theorem \ref{t:protagonist} for the case of Corollary \ref{c:Khh-sat}.
\end{proof} 

Next, we turn to proving convergence of environments around geodesics. For this, we introduce the following notation. Consider a geodesic $\pi$ from an initial condition $(f, s)$ to a final condition $(g, t)$.
For the purposes of this definition we assume the initial and final condition $f, g$ satisfy the conditions of Lemma \ref{l:basic}. First, for $r \in (s, t)$, define the time-evolved initial and final conditions
$$
\cL(f,s;x,r)=\max_{z \in \R} \Big(f(z) +  \cL(z, s; x, r)\Big) \quad \text{ and } \quad \cL(x,r;g,t)= \max_{y \in \R} \Big(\cL(x, r; y, t) + g(y)\Big).
$$
Now for $r \in (s, t)$ and a scale $\e> 0$, let $X_\e$ be location of the maximum of the function $\cL(f,s;\cdot,r) + \cL(\cdot,r+\e^3;g,t)$. This maximum is almost surely unique by Lemma \ref{l:basic}. To apply this lemma, we use the independent increment and property of $\cL$ to get that $\cL(f,s;\cdot,r) + \cL(\cdot,r+\e^3;g,t) \eqd \cL(f,s;\cdot,r) + \cL(\cdot,r;g,t - \ep^3)$, which satisfies the assumptions of the lemma after applying scale and translation invariance of $\cL$.
The point $X_\e$ is the approximate location of $\pi$ in the time window $[r, r+\e^3]$. Then define
\begin{align*}
F_\epsilon(z)&=\epsilon^{-1}
(
\cL(f,s;X_\e + \e^2 z,r)
-
\cL(f,s;X_\e,r)), 
\qquad 
&&z \in \R,\\
G_\e(z) &= \epsilon^{-1}
(
\cL(X_\e + \e^2 z,r+\e^3;g,t)
-
\cL(X_\e,r+\e^3;g,t))
, 
\qquad &&z \in \R,\\
\cL_\epsilon(z,s';y,t') &= \epsilon^{-1}\mathcal L(X_\e+ \epsilon^2 z, r + \epsilon^3 s'; X_\e+ \epsilon^2 y, r + \epsilon^3 t'), \qquad  &&z,y\in \mathbb R, \;0\le s'< t' \le 1,  
\\
\pi_\e(s') &= \e^{-2}(\pi(r + s' \e^3) - X_\e), \quad &&s' \in [0, 1],
\\
W_\e(s') &= \e^{-1}\cL(\pi(r), r; \pi(r + s' \e^3), r + s'\e^3), \quad &&s' \in [0, 1].
\end{align*}
We call the quintuple
\begin{equation}\label{E:environment}
Z_\e = (F_\e, G_\e, \cL_\e, \pi_\e, W_\e)
\end{equation}
the \textbf{environment around $(\pi, r)$ at scale $\e$}. This coincides with the definition in the introduction when $f, g$ are equal to $0$ at single points $x, y$, and are equal to $-\infty$ elsewhere, and $\pi$ is the geodesic from $(x, s)$ to $(y, t)$. In our environment notation, the dependence on $\pi, r$ is suppressed. These values will be indicated with square brackets, e.g. $Z_\e[\pi, r]$, if they are not clear from context. 

Our first step in proving Theorem \ref{t:defgamma} is a straightforward corollary of Theorem \ref{t:brbess}.

\begin{corollary}\label{p:bbldecomp}Let $t>1$, and let $f, g$ satisfy the conditions of Corollary \ref{c:easy}. For each $\e \in (0, 1)$, let $\ga^\e$ be the a.s. unique geodesic from $(f, 0)$ to $(g, t + \e^3)$. 
	Then as $\e\to 0$,
	$$(F_\e,\; G_\e, \;\mathcal L_\epsilon)[\ga^\e, 1] \cvgd (-R+B,\; -R-B,\; \mathcal L)\,.$$
	Also, for any $K>0$ and $0<\e<1$, on the event that $X_\e\in [-K,K]$, on the compact set $[-K\e^{-2},K\e^{-2}]^4$ we have
	$$(F_\e,\; G_\e, \;\cL_\epsilon(\cdot,0;\cdot,1))[\ga^\e, 1]\ll (-R+B,\; -R-B,\; \mathcal L(\cdot,0;\cdot,1)).$$
	Moreover, for any fixed $K > 0$ the Radon-Nikodym derivatives are tight for all $0<\e<1$. Here we use the notation $X\ll Y$ for random variables $X, Y$ when the law of $X$ is absolutely continuous with respect to the law of $Y$.
\end{corollary}

\begin{proof} We suppress the notation $[\ga^\e, 1]$ in the proof. First observe that $(F_\epsilon,G_\e)$ and $\cL_\e$ are independent because of the independent increment property of the directed landscape. Now, we condition on $F_\epsilon, G_\epsilon$. Then $X_\e$ is fixed and by the spatial stationarity and rescaling property of directed landscape in Lemma \ref{l:invariance}, we have that $\mathcal L_\epsilon\overset{d}{=}\mathcal L$.  Therefore it suffices to show that $(F_\ep, G_\ep) \cvgd (-R + B, -R - B)$ with tight Radon-Nikodym derivatives on every compact set $[-K\e^{-2},K\e^{-2}]\times [-K\e^{-2},K\e^{-2}]$.  This follows from Theorem \ref{t:brbess}, 
	since the time-evolved initial and final conditions $(\cL(f,0;\cdot,1),\cL(\cdot,1+\e^3,g,t+\e^3))$ that give rise to $F_\e, G_\e$ are equal in distribution to $(\cL(f,0;\cdot,1),\cL(\cdot,1,g,t))$ for all $\e$ by the independent increment property and time stationarity of $\cL$.
\end{proof}

To move from convergence of $(F_\e, G_\e, \cL_\e)$ to convergence of the whole environment, we need a tightness result.

\begin{lemma}\label{l:tight}  In the setup of Corollary \ref{p:bbldecomp} with $X_\e = X_\e[\ga^\e, 1]$, the set of random variables $\left\{\epsilon^{-2}(X_\e-\gamma^\e(1))\right\}$ and $\left\{\epsilon^{-2}(X_\e-\gamma^\e(1+\e^3))\right\}$ are {tight} for $\e\in (0,1)$.
	
\end{lemma}

\begin{proof} We want to show that for any $\delta>0$, there exists $M>0$ such that
	\begin{equation}\label{e:tight}
	\P(\epsilon^{-2}|X_\e-\gamma^\e(1)|>M)<\delta\,
	\end{equation}
	for all $\e \in (0,1)$.
	First choose $K>0$ such that for all $\e>0$, both $X_\e$ and $\gamma^\e(1)$ are in $[-K,K]$ with probability at least $1-\delta/2$, and call this event $A_\e$.
	Observe that
	\[(\epsilon^{-2}(\gamma^\e(1)-X_\e),\epsilon^{-2}(\gamma^\e(1+\e^3)-X_\e))=\argmax_{x,y}\{F_\e(x)+\cL_\e(x,0;y,1)+G_\e(y)\}\,.\]
	Moreover, by Corollary \ref{p:bbldecomp}, for all $0<\e<1$, on $[-K\e^{-2},K\e^{-2}]^4$ and on the event that $X_\e\in [-K,K]$,
	\begin{equation}\label{e:randvectabscont}
	(F_\e,\; G_\e, \;\mathcal L_\epsilon(\cdot,0;\cdot,1)))\ll (-R+B,\; -R-B,\; \mathcal L(\cdot,0;\cdot,1))
	\end{equation}
	with Radon-Nikodym derivatives that are tight. If $\mu_\e$ and $\nu$ denote the laws of the random vectors on the left-hand side and the right-hand side respectively in \eqref{e:randvectabscont}, then for the given $\delta>0$, there exists $\delta'>0$ such that for all $\e\in(0, 1)$ and any event $A$,
	\[\nu(A)<\delta'\implies \mu_\e(A)<\delta/2\,.\]
	Now define
	\[(X,Y):=\argmax_{x,y}\{B(x)-R(x)+\mathcal {L}(x,0;y,1)-B(y)-R(y)\}.\]
	This exists and is a.s. unique by Corollary \ref{c:Khh-sat}.
	Next choose $M>0$ such that
	\[\P(|X|>M)<\delta'\,.\]
	Hence
	\[\P(\{\epsilon^{-2}|X_\e-\gamma^\e(1)|>M\}\cap A_\e)<\delta/2\,.\]
	This, together with the bound on the probability of $A_\e$, implies \eqref{e:tight} and hence proves the tightness for $\left\{\epsilon^{-2}(X_\e-\gamma^\e(1))\right\}$. The tightness for $\left\{\epsilon^{-2}(X_\e-\gamma^\e(1+\e^3))\right\}$ is similar.
\end{proof}

\begin{theorem}\label{t:mainconv}
	In the setup of Corollary \ref{p:bbldecomp}, we have
	the joint convergence
	\begin{equation*}
	Z_\e[\ga^\e, 1] = (F_{\e_n},\; G_{\e_n}, \;\mathcal L_{\e_n},\; \ga^\e_\e,\; W_\e) \cvgd Z := (B-R, - B - R, \cL, \Ga, W_\Ga). 
	\end{equation*}
	Here the underlying topology is uniform convergence on compact subsets of the domain of $Z_\e$. That is, we have uniform convergence on compact subsets of $\R$ for the evolved initial conditions $F_\e, G_\e$, uniform convergence on $[0, 1]$ for the rescaled geodesic and weight function $\ga^\e_\e, W_\e$, and uniform convergence on compact subsets of $\{(x,s; y, t) \in \Rd: s, t\in [0, 1]\}$ for $\cL_\e$.
\end{theorem}

\begin{proof} 
	So we can apply Skorokhod's representation theorem, we will work with countable sets of $\e \in (0, 1)$.
	To prove the theorem, it suffices to show that for any sequence $\e_n \to 0$, that $Z_{\e_n} \cvgd Z$.
	First, by Corollary \ref{p:bbldecomp} and Skorokhod's representation theorem, we can find a new coupling of $F_{\e_n},\; G_{\e_n}, \;\mathcal L_{\e_n}, R, B, \cL$ such that almost surely
	$$
	(F_{\e_n},\; G_{\e_n}, \;\mathcal L_{\e_n}) \to (-R+B,\; -R-B,\; \mathcal L)
	$$
	as $n \to \infty$
	in the uniform-on-compact topology.
	Moreover, since $\cL_\e \eqd \cL$, $\cL$ is independent of $(R, B)$ and $\cL_\e$ is independent of $(F_\e, G_\e)$ for all $\e$, we may assume that $\cL_{\e_n} = \cL$ for all $n$. Next, for a function $h:\R\to \R$, we set $h^b$ to be equal to $h$ on $[-b, b]$ and equal to $-\infty$ elsewhere. Let $\Ga^b$ be the almost surely unique $\cL$-geodesic from $((B- R)^b, 0)$ to $((-B- R)^b, 1)$, and let $\Ga^b_{\e_n}$ be the almost surely unique $\cL$-geodesic from $(F^b_{\e_n}, 0)$ to $(G_{\e_n}^b, 1)$. Corollary \ref{C:precompact} ensures that on $\Om$, the set of geodesic graphs 
	$$
	\{\mathfrak{g} \Ga^b_{\e_n} : n \in \N\}
	$$
	is precompact for each $b$. Let $\Ga'$ be any subsequential limit point. Continuity of $\cL$ and almost sure uniform convergence of $(F_{\e_n}, G_{\e_n}) \to (B-R, -B - R)$ on $[-b, b]^2$ implies that 
	$$
	F_\e(\Ga^b_{\e_n}(0)) + \ell(\Ga^b_{\e_n}) + G_\e(\Ga^b_{\e_n}(1)) \to [B - R](\Ga'(0)) + \ell(\Ga') + [-B - R](\Ga'(1))
	$$
	almost surely along this subsequence. Now, since each $\Ga^b_{\e_n}$ is a geodesic from $(F^b_{\e_n}, 0)$ to $(G_{\e_n}^b, 1)$, the left-hand side above converges to 
	$$
	\max_{\pi} [(B - R)^b](\pi(0)) + \ell(\pi) + [(-B - R)^b](\pi(1)),
	$$
	where the max is over all paths $\pi:[0, 1] \to \R$.
	By Corollary \ref{c:Khh-sat}, $\Ga^b$ is almost surely the unique path attaining this max and so $\Ga' = \Ga^b$ almost surely. Hausdorff convergence of the graphs $\mathfrak{g} \Ga^b_{\e_n}$ implies uniform convergence of functions, so $\Ga^b_{\e_n} \to \Ga^b$ uniformly almost surely. 
	
	Uniform continuity of $\cL$ on any compact set then implies that the weight functions 
	$
	W_{\Ga^b_{\e_n}} \to W_{\Ga^b}
	$
	converge uniformly almost surely on $[1/2, 1]$. Moreover, for any geodesic $\pi:[0, 1] \to \R$, we can use the fact that the triangle inequality is an equality along $\pi$ to write
	$$
	W_\pi(t) = \cL(\pi(0), 0; \pi(1), 1) - \cL(\pi(t), t; \pi(1), 1).
	$$
	This representation implies that $
	W_{\Ga^b_{\e_n}} \to W_{\Ga^b}
	$
	converge uniformly almost surely on $[0, 1/2]$ as well, and so $
	W_{\Ga^b_{\e_n}} \to W_{\Ga^b}
	$
	uniformly almost surely. In summary, we have shown 
	\begin{equation}
	\label{e:fullcvg'}
	(F_{\e_n}, G_{\e_n}, \cL_{\e_n}, \Ga^b_{\e_n}, W_{\Ga^b_{\e_n}}) \cvgd (B-R, - B - R, \cL, \Ga^b, W_{\Ga^b})
	\end{equation}
	as $n \to \infty$. Now, Theorem \ref{t:protagonist} and Lemma \ref{l:tight} imply that 
	$$
	\lim_{b \to \infty} \limsup_{n \to \infty} \P\Big( (\Ga^b_{\e_n}, W_{\Ga^b_{\e_n}}) \ne (\ga^{\e_n}_{\e_n}, W_{\e_n}) \Big) = 0, \quad \lim_{b \to \infty} \P\Big( (\Ga^b, W_{\Ga^b}) \ne (\Ga, W_{\Ga})\Big) =0,
	$$
	which combined with \eqref{e:fullcvg'} yields the theorem.
\end{proof}

\begin{corollary}
	\label{C:brownian-cor}
	Let $H_\e, G_\e:\R \to \R \cup \{-\infty\}$ be random initial conditions indexed by $\e$ that are independent of $\cL$ and such that $H_\e, G_\e = -\infty$ off of some closed interval $[-d, d]$ for all $\e$. Further suppose that on $[-d, d]$, that the pair $(H_\e - H_\e(0), G_\e - G_\e(0))$ is absolutely continuous with respect to two independent, variance $2$ Brownian motions on $[-d, d]$ and with tight Radon-Nikodym derivatives.
	
	Now let $t > 1$, and let $\pi^\e$ be the almost surely unique geodesic from $(H_\e, 0)$ to $(G_\e, t + \e^3)$. Then with $Z$ as in Theorem \ref{t:mainconv}, 
	$$
	Z_\e[\pi^\e, 1] \cvgd Z,
	$$
	in the topology of uniform convergence on compact sets.
	Moreover, $Z_\e$ is asymptotically independent of $H_\e, G_\e$ in the sense of \eqref{e:asind} below.
\end{corollary}

\begin{proof}
	We first check the theorem in the case when $H_\e, G_\e$ are independent Brownian motions for every $\e$. In this case, we can define a new coupling of the environments where $H_\e = B_1, G_\e = B_2$ for every $\e > 0$ for some Brownian motions $B_1, B_2$. Conditioned on $B_1 = h, B_2 = g$, we are in the setting of Theorem \ref{t:mainconv}. Therefore
	$$
	\P(Z_\e \in \cdot \; | \; B_1, B_2) \cvgd \P(Z \in \cdot).
	$$
	Here we think of the left-hand side as a random measure $\mu^\e_{B_1, B_2}$ and
	we think of the right-hand side as a random measure whose law is supported on a single measure $\mu^*$. The underlying topology for almost sure convergence is the weak topology on measures.
	Now, consider the case of general $H_\e, G_\e$. Since the Radon-Nikodym derivatives of $(H_\e, G_\e)$ against $(B_1, B_2)$ are tight, so are the Radon-Nikodym derivatives of the laws of $\mu^\e_{H^\e, G^\e}$ against $\mu^\e_{B_1, B_2}$. Therefore since  $\mu^\e_{B_1, B_2}$ converges to a random measure supported at a single point $\mu^*$ as $\e \to 0$, so does $\mu^\e_{H^\e, G^\e}$, that is
	\begin{equation}\label{e:asind}
	\P(Z_\e \in \cdot \; | \; H_\e, G_\e) \cvgd \P(Z \in \cdot).
	\end{equation}
\end{proof}

Finally we can show Theorem \ref{t:defgamma}. We restate the result here for the reader's convenience.
\begin{theorem}\label{t:vimp'} Let $\pi$ be the directed geodesic from $p = (0,0)$ to $q = (0,1)$, and consider times $t_\ep \in (0, 1), \e \in (0, 1)$ with
	\begin{equation}
	\label{E:e-de}
	\lim_{\e \to 0} \frac{\ep^3}{\min(t_\ep, 1 - t_\ep)} = 0.
	\end{equation}
	Then with $Z$ as in Theorem \ref{t:mainconv}, we have 
	\begin{equation*}
	Z_\e[\pi, t_\e] \cvgd Z.
	\end{equation*}
\end{theorem}

\begin{remark}
	\label{R:other-points}
	A similar theorem holds for any other choice of $(p, q) \in \Rd$ by the invariance properties in Lemma \ref{l:invariance}.
\end{remark}

\begin{proof}[Proof of Theorem \ref{t:vimp'}]
	Without loss of generality, we can assume that the sequence $t_\e$ converges to a time $t \in [0, 1]$. By time-reversal symmetry of $\cL$, we may assume $t \in [0, 1/2]$.
	For small enough $\e$, the interval $[t_\e/2, 3t_\e/2 + \e^3]$ is strictly contained in $(0, 1)$. For such $\e$, $\pi$ agrees with the almost surely unique geodesic $\pi^\e$ between the initial and final conditions
	$$
	(\cL(p; \cdot, t_\e/2), t_\e/2) \qquad \text{ and } \qquad (\cL(\cdot, 3t_\e/2 + \e^3; q), 3t_\e/2 + \e^3)
	$$
	on the interval $[t_\e/2, 3t_\e/2 + \e^3]$.	Now, by Lemma \ref{l:invariance} we have the invariance
	$$
	\cL(x, s; y, t) \eqd m \cL(m^{-2} x, m^{-3} s - 1; m^{-2} y, m^{-3} t - 1)
	$$
	jointly in all variables, with $m^3 = t_\e/2$. Therefore 
	\begin{equation}
	\label{E:Zee}
	Z_\e[\pi, t_\e] = Z_\e[\pi^\e, t_\e] \eqd Z_\de[\ga^\de, 1],
	\end{equation}
	where $\de^3 = 2 \e^3/t_\e$, $\ga^\de$ is the geodesic from $(H_\de, 0)$ to $(G_\de, 2 + \de^3)$, and $H_\de, G_\de$ are given by
	$$
	H_\de(x) = \cL(0,-1; x, 0), \qquad G_\de(x) = \cL(x, 2 + \de^3; 0, 2/t_\e - 1).
	$$
	Now, again using the notation $f^m$ for the function which is equal to $f$ on $[-m, m]$ and equal to $-\infty$ elsewhere, let $\ga^{\de, m}$ denote the geodesic from $(H^m_\de, 0)$ to $(G^m_\de, 2 + \de^3)$. The decay bound in Proposition \ref{P:infty-blowup} on $\cL$ ensures that
	\begin{equation}
	\label{E:minfty}
	\lim_{m \to \infty} \limsup_{\de \to 0} \P(Z_\de[\ga^{\de, m}, 1] \ne Z_\de[\ga^\de, 1]) = 0.
	\end{equation}
	Therefore by \eqref{E:Zee}, \eqref{E:minfty}, and the fact that $\e \to 0$ as $\de \to 0$ by \eqref{E:e-de}, it suffices to show that $Z_\de[\ga^{\de, m}, 1] \cvgd Z$ as $\de \to 0$ for all $m$. 
	
	Now, $H_\e$ is equal in distribution to a parabolic Airy process $\mathcal L_1$ for all $\de$, and $G_\de(\cdot) \eqd c_\de \mathcal L_1(\cdot/ c^2_\de )$, where the rescaling factors
	$$
	c_\de = \left( 2/t_\e - 3 - \de^3 \right)^{1/3} = \left(\frac{2 - 3t_\e - 2\e^3}{t_\e}\right)^{1/3} 
	$$
	are bounded below by $1/2$ for all small enough $\e$. Since $\cL_1$ is Brownian on compacts by Theorem \ref{t:airy2} and $H_\de, G_\de$ are independent, the pair $(H^m_\de, G^m_\de)$ are absolutely continuous with respect to two independent, variance $2$ Brownian motions on $[-m, m]$. Moreover, the lower bound $c_\de \ge 1/2$ ensures that the Radon-Nikodym derivatives of $(H^m_\de, G^m_\de)$ are tight in $\de$. Therefore $Z_\de[\ga^{\de, m}, 1] \cvgd Z$ by Corollary \ref{C:brownian-cor}.
\end{proof}

\begin{remark}
	\label{R:asym-later}
	The application of Corollary \ref{C:brownian-cor} at the end of the proof of Theorem \ref{t:vimp'} shows that $Z_\e[\pi, t_\e]$ is asymptotically independent of $\cL(p; x, t_\e/2)$ and $\cL(x, 3t_\e/2 + \e^3; q)$ in the sense of \eqref{e:asind}. This will be used in the next section to help establish asymptotic independence of geodesic increments.
\end{remark}

\section{Asymptotic independence of increments}

In this section, we establish that local environments and increments around the directed geodesic are asymptotically independent as the length of the time interval decreases. The proof is based on a resampling argument.

For a directed landscape $\cL$ and a time interval $I = [r, r']$, we form a new \textbf{resampled landscape} $\cL^I$ in the following way. For $(x, s; y, t) \in \Rd$ with $(s, t) \cap [r, r'] = \emptyset$, set $\cL^I = \cL$. For $(x, s; y, t) \in \Rd$ with $[s, t] \sset [r, r']$, let $\cL^I = \cL'$, where $\cL'$ is a new directed landscape independent of $\cL$. This defines $\cL^I$ uniquely by the metric composition law in Theorem \ref{t:L-unique}.
We start with a few resampling lemmas.

\begin{lemma} \label{L:coal-2}
	Let $\pi$ be the almost surely unique geodesic in $\cL$ from $(0,0)$ to $(0, 1)$, and let $I_n = [r_n, s_n] \sset [0, 1]$ be any sequence of intervals such that $r_n, s_n$ converge to the same point $s \in [0, 1]$.
	
	For any $t \in (s,1]$, there exists a random open neighborhood $V_t \sset \R \times [(t + s)/2, 2]$ of $(\pi(t),t)$ so that 
	any $\cL$-geodesic $\gamma$ from $(0,0)$ to a point in the closure $\bar V_t$ satisfies $\gamma(s) =\pi(s)$. Moreover, for all such $V_t$, and any open  neighborhood $U$ of  $(\pi(s), s)$, as $n\to \infty$, 		\begin{equation}\label{E:inprob}
	\P\Big((\gamma(s_n), s_n) \in U \text{ for all } \cL^{I_n}\text{-geodesics $\gamma$ from } (0,0) \text{ to a point in } V_t \Big) \to 1.
	\end{equation}
\end{lemma}

The reason in Lemma \ref{L:coal-2} for asking that $V_t \sset  \R \times [(t + s)/2, 2]$ is just to give separation between $\bar V_t$ and the strips $\R \times [r_n, s_n]$.
See Figure \ref{fig:Lemma71} for an illustration of the statement of Lemma \ref{L:coal-2}.
\begin{figure}
    \centering
    \includegraphics{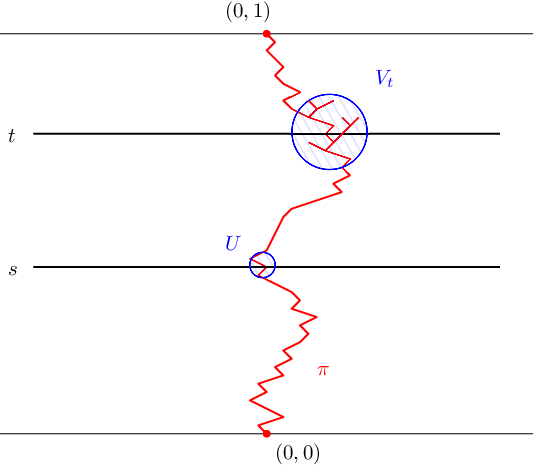}
    \caption{The main objects in Lemma \ref{L:coal-2}. For large $n$, with high probability all $\cL^{I_n}$-geodesics to a point in $V_t$ will be close to the point $(\pi(s), s)$ at the time $s_n$ and hence contained in $U$.}
    \label{fig:Lemma71}
\end{figure}

\begin{proof}
	By the uniqueness of $\pi$, the geodesic from $0:=(0,0)$ to any point $(\pi(t),t)$ is also unique. Lemma \ref{L:overlap}(ii) then implies the existence of the open neighborhood  $V_t$.  
	
	For the second claim, it suffices to show that for any subsequence $Y \sset \N$, we can find a further subsequence $Y' \sset Y$ where \eqref{E:inprob} holds if the limit is taken over $n \in Y'$.
	First, the tail bounds on $\cL$ in Proposition \ref{P:infty-blowup} imply that for any $b > 0$, we have
	\begin{equation}
	\label{E:tail-bounds}
	\sup \Big\{|\cL(x, r_n; y, s_n) - \cL^{I_n}(x, r_n; y, s_n)| : x, y \in [-b, b] \Big\} \cvgp 0.
	\end{equation}
	Next, Lemma \ref{L:geod-tight} implies that for each $n$, we can find minimal random intervals $[-B_n, B_n], [-B, B]$ such that
	\begin{align*}
	\{\mathfrak{g} \gamma : \gamma \text{ is an } \cL^{I_n}\text{-geodesic from $0$ to a point in } V_t\} &\sset [-B_n, B_n] \times [0, 1]\\
	\{\mathfrak{g} \gamma : \gamma \text{ is an } \cL\text{-geodesic from $0$ to a point in } V_t\} &\sset [-B, B] \times [0, 1].
	\end{align*}
	Moreover, all the intervals $[-B_n, B_n]$ are equal in distribution, since all the resampled landscapes are. Therefore \eqref{E:tail-bounds} also holds with the random interval $[-(B_n \vee B), B_n \vee B]$ in place of $[-b,b]$, and so along any subsequence $Y \sset \N$, we can find a further subsequence $Y' \sset Y$ such that 
	\begin{equation}
	\label{E:supcL}
	\sup \Big\{|\cL(x, r_n; y, s_n) - \cL^{I_n}(x, r_n; y, s_n)| : x, y \in [-(B_n \vee B), B_n \vee B] \Big\} \to 0
	\end{equation}
	almost surely along $Y'$. From now on, we work on the subsequence $Y'$, and the event $\Sigma$ with $\P\Sigma=1$ where \eqref{E:supcL} holds.
	
	To prove that \eqref{E:inprob} holds along $Y'$, it suffices to show that for any sequence of $\cL^{I_n}$-geodesics $\gamma_n$ from $0$ ending at points $q_n \in V_t$, that $(\gamma_n(s_n), s_n) \in U$ for all large enough $n$.  
	Let $\sig_n$ be an $\cL$-geodesic from $0$ to $q_n$, and use the shorthand $\sig_n(r)$ for $(\sig_n(r), r)$.  By the geodesic property,  we have
	\begin{align}
	\nonumber
	\cL(0; q_n) &= \cL(0;   \sig_n(r_n)) + \cL( \sig_n(r_n);  \sig_n(s_n)) + \cL( \sig_n(s_n); q_n).
	\end{align}
	Since $\cL = \cL^{I_n}$ for time intervals that do not intersect $[r_n, s_n]$, the right-hand side equals	
	$$
	\cL^{I_n}(0;  \sig_n(r_n)) + \cL( \sig_n(r_n);  \sig_n(s_n)) + \cL^{I_n}( \sig_n(s_n); q_n) .
	$$
	By \eqref{E:supcL} and the fact that $\sig_n$ stays in the interval $[-B, B]$, this equals 
	\begin{equation}
	\cL^{I_n}(0;  \sig_n(r_n)) + \cL^{I_n}( \sig_n(r_n);  \sig_n(s_n)) + o(1) + \cL^{I_n}( \sig_n(s_n); q_n)
	\nonumber 
	\le \cL^{I_n}(0; q_n) + o(1).
	\end{equation}
	The inequality above follows from the triangle inequality for $\cL^{I_n}$.
	Similarly,
	\begin{align}
	\nonumber 
	\cL^{I_n}(0;  \gamma_n(s_n)) 
	\le \cL(0;  \gamma_n(s_n)) + o(1).
	\end{align}
	Therefore, by the geodesic property of $\gamma_n$ we have 
	\begin{align}
	\nonumber
	\cL(0; q_n) &\le \cL^{I_n}(0; q_n) + o(1) \\
	\nonumber
	&= \cL^{I_n}(0;  \gamma_n(s_n)) + \cL( \gamma_n(s_n); q_n) +o(1) \\
	\label{E:Ltriback}
	&\le \cL(0;  \gamma_n(s_n)) + \cL( \gamma_n(s_n); q_n) + o(1).
	\end{align}
	We use this inequality to show that $\ga_n(s_n) \to \pi(s).$ Since the $q_n$ are contained in the compact set $\bar V_t$, it suffices to show that $\ga_n(s_n) \to \pi(s)$ along any subsequence where the points $q_n$ converge to some $q \in \bar V_t$. On such a subsequence, $\cL(0; q_n) \to \cL(0; q)$ by continuity, so \eqref{E:Ltriback} and Proposition \ref{P:infty-blowup} imply that the sequence $\ga_n(s_n)$ is precompact and any subsequential limit of $(\ga_n(s_n), s_n)$ lies on an $\cL$-geodesic from $0$ to $q$. Since any $\cL$-geodesic from $0$ to a point in $\bar V_t$ agrees with $\pi$ at time $s$, we have $\gamma_n(s_n) \to \pi(s)$. Hence $(\gamma_n(s_n), s_n) \in U$ for all large enough $n$, as desired, which shows \eqref{E:inprob} along the subsequence $Y'$.
\end{proof} 

\begin{figure}
    \centering
    \includegraphics{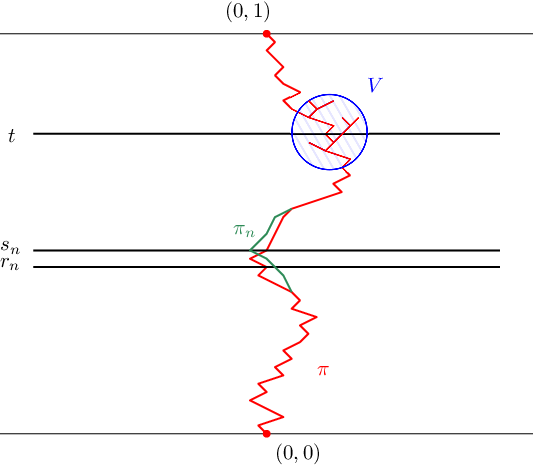}
    \caption{An illustration of Proposition \ref{p:coal-1}. In the resampled landscapes $\cL^{I_n}$, the geodesics to $(0, 1)$ or to points in $V$ will differ from the corresponding $\cL$-geodesics, but only in a shrinking interval around $s$. The fact that the geodesics trees from $(0, 0)$ to $V \cup \{0, 1\}$ typically agree in $\cL$ and $\cL^{I_n}$ for large $n$ implies the local equality of the environments $Z_\ep[\pi_n, t_\ep]$ and $Z[\pi, t_\ep]$ in Corollary \ref{c:coal-2}.}
    \label{fig:Lemma72}
\end{figure}
\begin{proposition}
	\label{p:coal-1} Let $\pi$ be the almost surely unique geodesic in $\cL$ from $(0,0)$ to $(0, 1)$, let $I_n = [r_n, s_n] \sset [0, 1]$ be any sequence of intervals such that $r_n, s_n$ converge to the same point $s \in [0, 1]$, and similarly define $\pi_n$ from each $\cL^{I_n}$.
	Let
	$$
	T_n = \{t \in [0, 1] : \pi(t) \ne \pi_n(t) \}.
	$$
	Then almost surely, $T_n \cup I_n$ is an interval for all $n$, and the length of $T_n \cup I_n$ converges to $0$ in probability as $n \to \infty$.
	
	Moreover, let $q = (\pi(t), t)$ be any point with $t \in [0, 1]$ and $t \ne s$. Then if $t > 0$, there exists an $\cL$-measurable open set $V \sset \R^2$ containing $q$ such that as $n \to \infty,$
	\begin{equation}\label{E:cLrg}
	\P\Big( \cL^{I_n}(0,0; m) - \cL^{I_n}(0,0; q) =  \cL(0,0; m) - \cL(0,0; q) \qquad \text{ for all $m \in V$} \Big) \to 1.
	\end{equation}
	Similarly, if $t < 1$, then there exists an $\cL$-measurable open set $U \sset \R^2$ containing $q$ such that as $n \to \infty$,
	\begin{equation}
	\label{E:lfCL}
	\P \Big(\cL^{I_n}(m; 0, 1) - \cL^{I_n}(q; 0, 1) =  \cL(m; 0, 1) - \cL(q; 0, 1) \qquad \text{ for all $m \in V$} \Big) \to 1.
	\end{equation}
\end{proposition}

See Figure \ref{fig:Lemma72} for an illustration of the statement and main proof idea in Proposition \ref{p:coal-1}.

\begin{proof}
	
	We prove the proposition for $s < t$. The $t > s$ case follows by time-reversal symmetry of $\cL$.
	
	On the interval $[0, r_n]$, $\pi$ and $\pi_n$ are $\cL$-geodesics from $0:=(0,0)$ to $\pi(r_n)$ and $\pi_n(r_n)$, respectively. Therefore since $\pi$ is almost surely unique, the overlap of $\pi$ and $\pi_n$ on $[0, r_n]$ is almost surely an interval containing $0$, see Lemma \ref{L:overlap}.
	Therefore
	$T_n \cap [0, r_n] = (R_n, r_n]$ for some $R_n$. Similarly, almost surely $T_n \cap [s_n, 1] = [s_n, S_n)$. Hence $T_n \cup I_n$ is an interval for all $n$. 
	
	
	To complete the proof of convergence we just need to show that $R_n, S_n \cvgp s$. On the interval $[s_n, 1]$, $\pi_n$ is an $\cL$-geodesic from $(\pi_n(s_n), s_n)$ to $(0,1)$. Also, $(\pi_n(s_n), s_n) \to (\pi(s), s)$ in probability as $n \to \infty$ by Lemma \ref{L:coal-2} in the case when $t=1$. Therefore by Lemma \ref{L:overlap}, $S_n \cvgp s$. By a symmetric argument, $R_n\cvgp s$. 
	
	We now prove \eqref{E:cLrg}. First, by Lemma \ref{L:overlap} again, we can find open sets $U, V\sset \R^2$ with $(\pi(s), s)\in U$ and $(\pi(t), t) \in V$ such that for any $\cL$-geodesic $\gamma$ from a point in $U$ to a point in $V$, we have 
	\begin{equation}
	\label{E:tss}
	\gamma\lf(\frac{t + s}{2}\rg) = \pi \lf(\frac{t + s}{2}\rg).
	\end{equation}
	Now let $V_t$ be as in Lemma \ref{L:coal-2}. By that lemma, the probability that all $\cL^{I_n}$-geodesics from $0$ to $V \cap V_t$ are in $U$ at time $s_n$ converges to $1$ with $n$. Since $\cL^{I_n}$-geodesics agree with $\cL$-geodesics after time $S_n$, this implies that $\P A_n \to 1$ as $n \to \infty$, where
	\begin{equation}
	\label{E:event}
	A_n = \lf\{\text{all $\cL^{I_n}$-geodesics $\gamma$ from $0$ to a point in $V \cap V_t$ satisfy \eqref{E:tss}} \rg\}.
	\end{equation}
	All $\cL$-geodesics from $0$ to $V_t \cap V$ also satisfy \eqref{E:tss} by the definitions of $V, V_t$. Therefore letting $w = (\pi((t+s)/2), (t+s)/2)$, on $A_n$, for $p, p' \in V_t \cap V$ we have
	$$
	\cL^{I_n}(0; p) - \cL^{I_n}(0; p') = \cL^{I_n}(w; p) - \cL^{I_n}(w; p') = \cL(w; p) - \cL(w; p') = \cL(0; p) - \cL(0; p'),
	$$
	yielding \eqref{E:cLrg}. In the $s < t$ case, claim \eqref{E:lfCL} is trivial.
\end{proof}

Proposition \ref{p:coal-1} implies that environments along resampled geodesics are close to environments along the original geodesic. For this corollary and for the remainder of the section, we use the environment notation \eqref{E:environment} of Section \ref{S:environment}.

\begin{corollary}
	\label{C:environ-close}
	Let $\pi$ be the almost surely unique geodesic in $\cL$ from $(0,0)$ to $(0, 1)$, let $I_n = [r_n, s_n] \sset [0, 1]$ be any sequence of intervals such that $r_n, s_n$ converge to the same point $s \in [0, 1]$, and similarly define $\pi_n$ from each $\cL^{I_n}$. Let $\{0 < t_\e < 1 :\e \in (0, 1)\}$ be a sequence converging to $t \ne s$ and satisfying condition \eqref{E:e-de}. Then for any compact set $K$, we have
	$$
	\lim_{n \to \infty} \limsup_{\e \to 0} \P(Z_\e[\pi_n, t_\e] = Z_\e[\pi, t_\e] \text{ on $K$}) = 1.
	$$
	Here the environment $Z_\e[\pi_n, t_\e]$ is defined in the resampled landscape $\cL^{I_n}$, whereas $Z_\e[\pi, t_\e]$ is defined in $\cL$, see \eqref{E:environment}.
\end{corollary}

Figure \ref{fig:Lemma72} is helpful for understanding the statement and proof of Corollary \ref{C:environ-close}. The only reason we require condition \eqref{E:e-de} on $t_\ep$ is to guarantee that the environments $Z_\ep$ are well-defined.

\begin{proof}
	We assume $s < t$, as the $t > s$ case is symmetric. As in Section \ref{S:environment}, let $X_\e$ be the almost surely unique maximizer of the function
	$$
	x \mapsto H_\e(x) := \cL(0,0; x, t_\e) + \cL(x, t_\e + \e^3; 0, 1),
	$$ 
	and similarly define $X_{\e, n}, H_{\e, n}$ with $\cL^{I_n}$ in place of $\cL$. Lemma \ref{l:tight} and invariance properties of $\cL$ (Lemma \ref{l:invariance}) imply that as $\e \to 0$, 
	$$
	|X_\e - \pi(t_\e)| \qquad \text{ and } \qquad |X_{\e, n} - \pi_n(t_\e)| \cvgp 0
	$$
	for every fixed $n \in \N$.
	Therefore by Proposition \ref{p:coal-1} and continuity of $\pi$, for any $\de > 0$, we have 
	\begin{equation}
	\label{E:Xddd}
	\lim_{n \to \infty} \limsup_{\e \to 0} \P(|X_\e - \pi(t)| > \de \text{ or } |X_{\e, n} - \pi(t)| > \de) = 0.
	\end{equation}
	Now also by Proposition \ref{p:coal-1}, there exists a random open set $V \sset \R^2$ containing $q = (\pi(t), t)$ such that with $0 :=(0,0)$, we have
	\begin{equation}
	\label{E:ee}
	\lim_{n \to \infty} \P\lf(\cL(0; m) - \cL(0;q) = \cL^{I_n}(0; m) - \cL^{I_n}(0; q) \text{ for all } m \in V \rg) = 1.
	\end{equation}
	Moreover, since $\cL$ and $\cL^{I_n}$ agree off of time intervals that intersect $I_n$, for some $n_0, \e_0$ we have
	\begin{equation}
	\label{E:cLXX}
	\cL(x, t_\e + \e^3; 0, 1) = \cL^{I_n}(x, t_\e + \e^3; 0, 1)
	\end{equation}
	for all $\e \le \e_0, n \ge n_0$. Putting together \eqref{E:ee} and \eqref{E:cLXX} implies that there exists a random interval $J = (a, b)$ containing $\pi(t)$ such that
	$$
	\lim_{n \to \infty} \limsup_{\e \to 0}\P( H_{\e, n} - H_\e \text{ is constant on } J) = 1,
	$$
	which combined with the tightness \eqref{E:Xddd} implies
	\begin{equation}
	\label{E:ee2}
	\lim_{n \to \infty} \limsup_{\e \to 0}\P( X_{\e, n} = X_\e) = 1.
	\end{equation}
	Equations \eqref{E:ee}, \eqref{E:cLXX}, and \eqref{E:ee2} then imply that for any compact set $K \sset \R$, we have
	\begin{equation*}
	\lim_{n \to \infty} \limsup_{\e \to 0}\P\Big( (F_\e, G_\e, \cL_\e)[\pi_n, t_\e] =  (F_\e, G_\e, \cL_\e)[\pi, t_\e] \text{ on $K \times \R \times D$} \Big) = 1,
	\end{equation*}
	where $D = \{(x, s; y, t) \in \R^4 : 0 \le s < t \le 1\}$ is the natural domain of the functions $\cL_\e$. To extend this to the full convergence of the environments, we just need to recognize that by the first part of Proposition \ref{p:coal-1},
	\[
	\lim_{n \to \infty} \P\Big( \pi(r) = \pi_n(r) \quad \text{ for all } \quad (t+s)/2 \le r \le 1\Big) = 1. \qedhere
	\]
\end{proof}

We can now prove asymptotic independence of environments at different locations along a geodesic. For this theorem and its proof, we use the environment notation \eqref{E:environment} of Section \ref{S:environment}.
\begin{theorem}
	\label{T:joint-cvg}
	Let $t_{1, \ep} <  \dots < t_{k, \ep} \in (0, 1)$ be points with $t_{i, \e}\to t_i$ as $\ep \to 0$ for some $t_1 <t_2 < \dots < t_k \in [0, 1]$.
	Suppose additionally that 
	$$
	\e^3/\min(t_{1, \e}, 1 - t_{k,\e}) \to 0 \quad \text{ as } \e \to 0.
	$$
	Let $\pi$ denote the $\cL$-geodesic from $(0,0)$ to $(0,1)$. Then we have
	$$
	(Z_\e[\pi, t_{1, \e}], \dots, Z_\e[\pi, t_{k, \e}]) \cvgd (Z_1, \dots, Z_k),
	$$
	where the $Z_i$ are independent copies of the limit object  $Z := (B-R, - B - R, \cL, \Ga, W_\Ga)$ first introduced in Theorem \ref{t:mainconv}.
\end{theorem}

\begin{proof}
	For each $i$, we have $Z_\e[\pi, t_{i, \e}] \cvgd Z_i$ by Theorem \ref{t:vimp'}. It remains to check asymptotic independence. We will check that $Z_\e[\pi, t_{1, \e}]$ is asymptotically independent of the remaining $k-1$ environments. The result then follows by induction.
	
	Define $I_n = [(t_1 - n^{-1}) \vee 0, t_1 + n^{-1}].$ Let $\pi_n$ be the geodesic from $(0,0)$ to $(0, 1)$ in the resampled landscape $\cL^{I_n}$. Since $\cL \eqd \cL^{I_n}$,
	$$
	(Z_\e[\pi, t_{j, \e}] : j =1, \dots, k) \eqd (Z_\e[\pi_n, t_{j, \e}] : j = 1, \dots, k)
	$$
	for all $n, \e$. Moreover, Corollary \ref{C:environ-close} guarantees that for any compact set $K$ in the domain of the $Z$'s, we have
	$$
	\lim_{n \to\infty} \limsup_{\e \to 0} \P(Z_\e[\pi, t_{j, \e}] \ne Z_\e[\pi_n, t_{j, \e}] \text{ on $K$ for some } j > 1) = 0.
	$$
	Therefore to complete the proof, it  suffices to show that for every fixed $n$ we have the asymptotic independence
	\begin{equation}
	\label{E:gane}
	\P( Z_\e[\pi_n, t_{1, \e}] \in \cdot \; | \; \cL) \cvgd \P( Z_1 \in \cdot).
	\end{equation}
	Here we think of both sides as random measures (the right-hand side is deterministic), and the underlying topology is weak convergence of measures. First, as in the proof of Theorem \ref{t:vimp'},
	$$
	Z_\e[\pi_n, t_{1, \e}] = Z_\e[\gamma_\e, t_{1, \e}]
	$$ 
	where $\gamma_\e$ is the geodesic from
	\begin{equation}
	\label{E:cLIn}
	(\cL^{I_n}(0,0; \cdot, t_{1, \e} - s_\e), t_{1, \e} - s_\e) \qquad \text{to} \qquad (\cL^{I_n}(\cdot, t_{1, \e} + s_\e; 0, 1), t_{1, \e} + s_\e).
	\end{equation}
	Here we choose $s_\e$ small enough so that $t_{1, \e} \pm s_\e \in I_n$ for all small enough $\e$, and
	$$
	\e^3/(t_{1, \e} - s_\e) \to 0, \qquad \e^3/s_\e \to 0.
	$$ 
	Since $t_{1, \e} \pm s_\e \in I_n$ for small enough $\e$, the environment $Z_\e[\gamma_\e, t_{1, \e}]$ is conditionally independent of $\cL$ given the boundary conditions \eqref{E:cLIn}. From here we can follow the proof of Theorem \ref{t:vimp'} verbatim to conclude that
	 \begin{equation}
	 \label{E:ganee}
	\P( Z_\e[\pi_n, t_{1, \e}] \in \cdot \; | \; \cL^{I_n}(0,0; \cdot, t_{1, \e} - s_\e), \cL^{I_n}(\cdot, t_{1, \e} + s_\e; q)) \cvgd \P( Z_1 \in \cdot)
	\end{equation}
	almost surely, which yields \eqref{E:gane} immediately by the conditional independence of $Z_\e[\pi_n, t_{1, \e}]$ and $\cL$ given $\cL^{I_n}(0,0; \cdot, t_{1, \e} - s_\e), \cL^{I_n}(\cdot, t_{1, \e} + s_\e; q)$. The fact that we get the convergence of the \emph{conditional} distributions in \eqref{E:ganee} rather than just the distributions of $Z_\e[\pi_n, t_{1, \e}]$ uses the asymptotic independence of $Z_\e[\pi_n, t_{1, \e}]$ and $\cL^{I_n}(0,0; \cdot, t_{1, \e} - s_\e), \cL^{I_n}(\cdot, t_{1, \e} + s_\e; q)$ noted in Remark \ref{R:asym-later}.
\end{proof}

\section{Uniform tail bounds for geodesic increments}
\label{S:prelimit-bds}

In this section we prove uniform tail bounds for  geodesic and weight function increments, Theorem \ref{t:geod-thm}. This will allow us to move from the distributional convergence in Theorem \ref{t:vimp'} and the asymptotic independence in Theorem \ref{T:joint-cvg} to convergence of moments. 

Recall that $\pi$ is the directed geodesic from $(0,0)$ to $(0,1)$. For $s < s + \e^3 \in [0, 1]$, define random variables
\begin{equation}
\label{E:Ise}
I_{s, \e} = \e^{-2}(\pi(s) - \pi(s + \e^3)), \qquad \text{ and } W_{s, \e} = \e^{-1}\cL(\pi(s), s; \pi(s + \e^3), s + \e^3).
\end{equation}
Theorem \ref{t:geod-thm} claims that 
\begin{equation}
\label{E:aIse}
\E a^{|I_{s, \e}|^3} < c,  \qquad \E a^{|W_{s, \e}|^{3/2}} < c.
\end{equation}
where $a>1,c$ are uniform in $s,s+\eps\in[0,1]$.

The case when $\min(s, 1-s - \ep^2) \le \ep^3$ is almost immediate, and we start here. 

\begin{lemma}
	\label{L:easier}
	In the context of Theorem \ref{t:geod-thm}, there exists $a> 1, c>0$ such that for $\min(s, 1-s - \ep^2) \le \ep^3$ we have
	\begin{equation}
	\label{E:aIse'}
	\E a^{|I_{s, \e}|^3} < c, \qquad \E a^{|W_{s, \e}|^{3/2}} < c.
	\end{equation}
\end{lemma}

\begin{proof}
	By flip symmetry, Lemma \ref{l:invariance}.3, we may assume that $s \le \e^3$. By Corollary \ref{C:geod-locale} and a union bound, we have
	\begin{equation}
	\label{E:gaga}
	\P(\pi(s), \pi(s + \e^3) \in [-m(s+\e^3)^{2/3}, m(s + \e^3)^{2/3}]) \ge 1-  ce^{-dm^3} 
	\end{equation}
	for some $c, d > 0$. Since $(s+\e^3)^{2/3} < 2^{2/3} \e^2$ by assumption, this implies the first claim in \eqref{E:aIse'}. Lemma \ref{L:sheet-bd} and scale-invariance of $\cL$ then implies
	\begin{equation}
	\label{E:39}
	|\cL(\pi(s), s; \pi(s + \e^3), s + \e^3)| \le \e|I_{s, \e}|^2 + \e C + \ep\log^{2/3}(2 + \ep^{-2}|\pi(s)| + \ep^{-2}|\pi(s + \e^3)|),
	\end{equation}
	where $\P(C > m) \le ce^{-dm^{3/2}}$. The second claim in \eqref{E:aIse'} then follows from \eqref{E:gaga} and \eqref{E:39}.
\end{proof}

The case when $\min(s, 1-s - \ep^3) > \e^3$ is much more involved. By flip symmetry of $\cL$, Lemma \ref{l:invariance}.3, it suffices to consider the case where $\ep^3 < s \le 1 - s - \ep^3$. Also, by scale-invariance of $\cL$ and the metric composition law along $\pi$, it suffices to show the following. Recall that the Airy sheet $\S_\sigma$ of scale $\sigma$ is defined as  $\sigma \S(\cdot/\sigma^2)$, where $\S$ is a standard Airy sheet, and, similarly, a parabolic Airy process of scale $\sigma$ is defined as 
$\cL_\sigma(\cdot)  \eqd \sigma \cL_1(\cdot/\sigma^2)$. 

Let $\cL_1, \cL_\sigma$ be parabolic Airy processes of scale $1, \sigma$, and let $\S_\de$ be an Airy sheet of scale $\de$ with all three objects independent. Define
\begin{equation}
\label{E:to-optimize}
(X, Y) = \argmax_{x, y} \Big(\cL_1(x) + \S_\de(x, y) + \cL_\sigma(y)\Big).
\end{equation}
By dilating time by a factor of $s$, we see that $I_{s, \ep} \eqd X - Y, W_{s, \ep} \eqd \S_\de(X, Y)$, with $\de = \ep/s^{1/3} < 1$ and $\sig = [(1 - s - \ep^3)/s]^{1/3} \ge 1$.
Therefore the $\ep^3 < s \le 1 - s - \ep^3$
case of Theorem \ref{t:geod-thm} is equivalent to the statement that there exists $a> 1$ such that for all $\de \in (0, 1), \sigma \ge 1$, we have
\begin{eqnarray}
\label{E:translated1}
\E a^{|X - Y|^3/\de^6} &<& \infty, \\ \E a^{|\S_\de(X, Y)/\de|^{3/2}} &<& \infty.\label{E:translated2}
\end{eqnarray}
We first show a modulus of continuity bound for the $\cL_\sigma$ which is uniform in $\sigma\ge 1$.

\begin{lemma}
	\label{l:mod-cont-1}
	There are universal constants $c, d > 0$ such that for any $\sigma > 0$, and $x, y \in \R$ with $|x-y| \ge 1$, we have
	\begin{equation*}
	|\cL_\sigma(x) - \cL_\sigma(y) + (x^2 - y^2)/\sigma^3| \le \Big(C + c \log^{1/2}(1 + |x| + |y|)\Big) \sqrt{|x - y|},
	\end{equation*}
	where $\P(C > a) \le ce^{-d a^2}$ for all $a > 0$. 
\end{lemma}
\begin{proof}
	In the proof, $c, d>0$ are constants that may change from line to line, but do not depend on any parameters.
	Let $\mathcal A_\sigma(x) = \cL_\sigma(x) + x^2/\sigma^3$. By Lemma \ref{L:airy-tails} and the Brownian scaling, for any $x, y \in \R, \sigma , a > 0$ we have 
	\begin{equation}
	\label{E:AtAy}
	\P(|\mathcal A_\sigma(x) - \mathcal A_\sigma(y)| \ge a \sqrt{|y - x|}) \le ce^{-da^2}.
	\end{equation}
	Bounds such as this naturally give rise to modulus of continuity estimates analogous to those for Brownian motion by essentially the same proof. For example, Lemma 3.3 in \citep{DV} implies that with 
	$$
	C_n=\sup_{x, y \in [n, n+1]}|\mathcal A_\sigma(x) - \mathcal A_\sigma(y)|
	$$
	we have $\P(C_n>a)< ce^{-d a^2}$  for every $n \in \N$.
	Also, for $n \ne m \in \N$, set 
	$$
	C_{n, m} = \frac{|\mathcal A_\sigma(n) - \mathcal A_\sigma(m)|}{\sqrt{|y - x|}},
	$$
	so by the triangle inequality, for all $x, y \in \R$, we have
	$$
	|\mathcal A_\sigma(x) - \mathcal A_\sigma (y)| \le C_{\lfloor x \rfloor, \lfloor y \rfloor} \sqrt{|\lfloor x \rfloor - \lfloor y \rfloor|} +  C_{\lfloor x \rfloor} +C_{\lfloor y \rfloor}.
	$$
	When $|x - y| \ge 1$, this is bounded above by 
	$$
	(C_{\lfloor x \rfloor, \lfloor y \rfloor} +  C_{\lfloor x \rfloor} +C_{\lfloor y \rfloor})\sqrt{|x - y|}.
	$$
	Now, the tail bounds on $C_n$ and $C_{n, m}$ guarantee that the supremum
	$$
	C = \sup_{x, y \in \R} C_{\lfloor x \rfloor, \lfloor y \rfloor} +  C_{\lfloor x \rfloor} +C_{\lfloor y \rfloor} - c\log^{1/2} (1 + |x| + |y|)
	$$
	is finite, and also satisfies $\P(C > a) \le ce^{-d a^2}$. The lemma follows.
\end{proof}
We can now show \eqref{E:translated1} and \eqref{E:translated2} when the difference $|X - Y|$ is macroscopic.

\begin{lemma}
	\label{L:long-bound}
	For every large enough $b>0$ there is $c,a>1$ so that for all $\de\in (0, 1)$ and $\sigma \ge 1$, we have
	\begin{eqnarray}
	\label{E:long-bound1}
	\E \big[\,a^{|X - Y|^3/\de^6};\, |X - Y| >b\,\big] &<& c,\\  \qquad \E \big[\,a^{|\S_\de(X, Y)/\de|^{3/2}};\, |X - Y| > b\,\big] &<& c.
	\label{E:long-bound2}
	\end{eqnarray}
\end{lemma}

\begin{proof}
	Throughout the proof $c, d$ are positive constants that may change from line to line, may depend on $b$ but not on $\sigma$ or $\de$.
	First, by Corollary \ref{C:geod-locale}, we have
	\begin{equation}
	\label{E:XYbd}
	\P(|X| > m \quad \text{or} \quad |Y| > m) \le ce^{-dm^3}.
	\end{equation}
	Next, by Lemma \ref{L:sheet-bd}, we have
	\begin{equation}
	\label{E:Se-bd}
	\begin{split}
	|\S_\de(x, y) + \de^{-3} (x-y)^2|&\le \de\lf(C + c\log^{2/3}(2 + (|x| + |y|)/\de^2)\rg) \\
	&\le \de\lf(C + c\log (2\de^{2} + |x| + |y|)\rg) + c \de \log (\de^{-1}),
	\end{split}
	\end{equation}
	for all $x,y\in \R$, where $C$ satisfies 
	\begin{equation}\label{e:Ctail}
	\P(C > m) \le ce^{-dm^{3/2}}. 
	\end{equation}
	To go from the first line to the second in \eqref{E:Se-bd}, we have just replaced the $2/3$ power in the log with a $1$, at the expense of changing $c$, and expanded out the log. 
	We substitute $X,Y$ into \eqref{E:Se-bd} and use the  tail bounds \eqref{e:Ctail} and \eqref{E:XYbd}, to get
	\begin{equation}
	\label{E:Sebd-2}
	|\S_\de(X, Y) + \de^{-3} (X-Y)^2| \le \de C' + c \de \log (\de^{-1}),
	\end{equation} 
	where $C'$ satisfies the same tail bound as $C$, after possibly modifying the constants $c, d$.
	Now, set $Z=X$ if $|X|\le|Y|$ and $Z=Y$ otherwise. Substituting $Z,Z$  into \eqref{E:Se-bd} we get 
	\begin{equation}
	\label{E:Sebd-3}
	|\S_\de(Z, Z) | \le \de C' + c \de \log (\de^{-1}).
	\end{equation} 
	By Lemma \ref{l:mod-cont-1} and \eqref{E:XYbd}, we have
	\begin{equation}
	\label{E:cLct} 
	\begin{split}
	\max(\cL_1(X)  - \cL_1(Z), &\cL_\sigma(Y)  - \cL_\sigma(Z)) \\
	&\le (D + \log^{1/2}(1 + |X| + |Y|))\sqrt{|X - Y|} \\
	&\le D'\sqrt{|X - Y|}
	\end{split}
	\end{equation}
	where $D'$ satisfies
	\begin{equation}
	\label{e:D;'tail}
	\P(D' > m) \le ce^{-dm^{2}}.
	\end{equation}
	To get the bound in \eqref{E:cLct}, we were able to remove the parabolic terms in Lemma \ref{l:mod-cont-1} by using that $|Z| = \min (|X|,|Y|)$.
	Next, the optimizer $(X, Y)$ in \eqref{E:to-optimize} satisfies
	\begin{equation}\label{e:sexy}
	0\le \S_\de(X, Y) - \S_\de(Z, Z) + \cL_1(X) - \cL_1(Z) + \cL_\sigma(Y) - \cL_\sigma(Z).
	\end{equation}
	By \eqref{e:sexy}, \eqref{E:Sebd-2}, \eqref{E:Sebd-3} and \eqref{E:cLct}, we have
	\begin{equation}
	\label{E:two-Cs}
	\frac{(X - Y)^2}{\de^3} \le 2 c \de \log(\de^{-1}) + 2\de C' + D' \sqrt{|X - Y|}.
	\end{equation}
	For this inequality to hold, one of the three terms on the right side above must be at least $(X-Y)^2/3\de^3$. As long as $b$ is sufficiently large, this cannot be the case for the first term when $|X - Y| \ge b$.
	If the second term is at least $(X-Y)^2/3\de^3$, we have
	\begin{equation}
	\label{E:XYcC}
	|X-Y| \le \sqrt{6} \de^2 C'^{1/2}.
	\end{equation}
	If the third term is at least $(X-Y)^2/3\de^3$, then
	\begin{equation}
	\label{E:XYee}
	|X-Y| \le 3^{2/3} \de^{2} D'^{2/3},
	\end{equation}
	so combining \eqref{E:XYcC} and \eqref{E:XYee}, when $|X-Y| > b$ we have
	\begin{equation}
	\label{E:almost}
	\de^{-2} |X-Y|\le \max(\sqrt{6} C'^{1/2}, 3^{2/3} D'^{2/3}).
	\end{equation}
	The bound \eqref{E:long-bound1} follows from \eqref{E:almost}, \eqref{e:Ctail} and \eqref{e:D;'tail}. Next, by \eqref{E:Sebd-2}, we have
	$$
	|\de^{-1} \S_\de(X, Y)| \le \frac{(X-Y)^2}{\de^4} + C'+ c\log(\de^{-1}).
	$$
	When $|X-Y|>b$ for large enough $b$, then $c \log(\de^{-1}) < \de^{-4} (X-Y)^2$, and so the right-hand side above is bounded by $2\de^{-4} (X-Y)^2 + C'$.
	The bound \eqref{E:long-bound2} then follows from \eqref{E:long-bound1} and \eqref{e:Ctail}.
\end{proof}

To complement this with a bound when $X, Y$ are close together, we use the Radon-Nikodym derivative bound for Airy processes. The following lemma is the key estimate.

\begin{lemma}
	\label{L:XY}
	Let $\cL_1, \cL_\sigma$ be parabolic Airy processes of scale $1, \sigma$,  let $\S_\de$ be an Airy sheet of scale $\de$ with all three objects independent, and let  $b>0$. For  $n \in \Z$, let
	$$
	(X_n, Y_n) = \argmax_{x, y \in [nb, (n+2) b]} \Big(\cL_1(x) + \S_\de(x, y) + \cL_\sigma(y)\Big).
	$$
	Then there exists constants $a > 1, c, d > 0$ such that for all $n \in \Z, \de \in (0, 1),$ and $\sigma > 1$ we have
	$$
	\E a^{\de^{-6} |X_n - Y_n|^3} \le ce^{dn^2}, \qquad \E a^{|\S_\de(X_n, Y_n)/\de|^{3/2}} \le ce^{dn^2}.
	$$
\end{lemma}

\begin{proof} Throughout the proof, $c, d$ are positive constants that may change from line to line. They do not depend on $\de, n,$ or $\sigma$. First, for a triple of functions $(g, s, h)$ with $g:\R \to \R, s:\R^2 \to \R, h:\R \to \R$, define
	$$
	f(g, s, h) = \sup \{ |y^* - x^*| : (x^*, y^*) \in \argmax_{x, y \in \R} (g(x) + s(x, y) + h(y))\}.
	$$
	We can also define $f$ for functions whose domain is only part of $\R$ or $\R^2$. Let $T_n$ be the location of the maximum of a Brownian motion on $[nb , (n+2)b ],$ let $R$ be an independent $2$-sided Bessel process on $[- T_n, 2b  - T_n]$, and let $B$ be an independent Brownian motion on the same interval. Let $\nu$ be the law of $(B-R, -B-R, \S_\de)$. By Remark \ref{R:bounded-domain} and Brownian scaling, there is a constant $a > 1$ such that
	\begin{equation}
	\label{E:ae6}
	\int a^{\de^{-6} f^3} d\nu \le c
	\end{equation}
	for all $\de \in (0, 1)$. Now let $(B_1, B_2)$ be independent Brownian motions of variance $2$ on $[nb , (n+2)b ]$, and let $T'$ be the location of the maximum of  $B_1 + B_2$. Letting
	\begin{align*}
	  -\tilde R(\cdot) &= \frac{1}{2}(B_1 + B_2)(T + \cdot) - \frac{1}{2}(B_1 + B_2)(T), \\
	  \tilde B(\cdot) &= \frac{1}{2}(B_1 - B_2)(T + \cdot) - \frac{1}{2}(B_1 - B_2)(T), \qquad \text{and}\\
	  \tilde S_\de(\cdot) &= \S_\de(T' + \cdot, T + \cdot)
	\end{align*}
	we can see that all three of these objects are independent, that $\tilde S_\de \eqd S_\de$ by shift invariance (see Lemma \ref{l:invariance}.2), $\tilde B \eqd B,$ and the law of $\tilde R$ is absolutely continuous with respect to the law of $R$ with Radon-Nikodym derivative in $L^2$ by Theorem \ref{l:bessel}. Therefore the law $\mu$ of
	\begin{equation}
	\label{E:B1b2}
	(B_1(T' + \cdot) - B_1(T'), B_2(T' + \cdot) - B_2(T'), \S_\de(T' + \cdot, T + \cdot))
	\end{equation}
	is absolutely continuous with respect to $\nu$ with Radon-Nikodym derivative $d\mu/d\nu$ in $L^2(\nu)$.

	Moreover, if we replace the Brownian motions $B_1, B_2$ with Brownian motions with drift $\De_1, \De_2$, and call  $\mu_{\De_1, \De_2}$ the law of the new process \eqref{E:B1b2}, then we have
	\begin{equation}
	\label{E:nuDe}
	\int \lf(\frac{d\mu_{\De_1, \De_2}}{d \mu} \rg)^2 d \mu \le c e^{d (\De_1^2 + \De_2^2)}.
	\end{equation}
	The bound on the Radon-Nikodym derivative $d\mu_{\De_1, \De_2}/d\mu$ is standard, and comes from the Radon-Nikodym derivatives of the normal random variables $N(\De_i, 4b )$ with respect to $N(0, 4b )$.
	
	Now let $M$ be the maximum of $\cL_1 + \cL_\sigma$ on $[nb , (n+2)b ]$. The law $\xi_{n, \sig}$ of 
	$$
	(\cL_1|_{[nb , (n+2)b ]}(M + \cdot) - \cL_1(M), \cL_\sigma|_{[nb , (n+2)b ]}(M+\cdot) - \cL_\sigma(M), \S_\de)
	$$
	is absolutely continuous with respect to $\mu_{2n b , 2n b /\sigma^3}$. Moreover, by Theorem \ref{t:airy2}, the stationarity of $\cL_1(x) + x^2, \cL_\sigma(x) + x^2/\sigma^3$, we have
	\begin{equation}
	\label{E:signu}
	\int \lf(\frac{d \xi_{n, \sig}}{d \mu_{2n b , 2n b /\sigma^3}} \rg)^2 d \mu_{2n b , 2n b /\sigma^3}\le c.
	\end{equation}
	Here to apply Theorem \ref{t:airy2} we have used that $\cL_\sigma(x+ nb) + (x+ nb)^2/\sigma^3, x \in [0, 2b]$ is equal in distribution to $\sig[\cL_1(x/\sig^2) + (x/\sig^2)^2], x/\sig^2 \in [0, 2b/\sig^2]$. The fact that the intervals $[0, 2b/\sig^2]$ are contained in $[0, 2b]$ for all $\sig \ge 1$ allows us to bound the right-hand side of \eqref{E:signu} by a constant that does not depend on $\sig$.
	
	We can put together all the Radon-Nikodym derivative estimates with multiple applications of the Cauchy-Schwarz inequality to get the result. We have
	\begin{align*}
	\E a^{\de^{-6} |X_n - Y_n|^3} &= \int \frac{d \xi_{n, \sig}}{d \mu_{2n b , 2n b /\sigma^3}} a^{\de^{-6} f^3} d \mu_{2n b , 2n b /\sigma^3} 
	\\
	&\le c \lf(\int a^{2 \de^{-6} f^3} d \mu_{2n b , 2n b /\sigma^3} \rg)^{1/2} = c \lf(\int a^{2 \de^{-6} f^3} \frac{d\mu_{2n b , 2n b /\sigma^3}}{d \mu} d \mu \rg)^{1/2}, 
	\end{align*}
	by \eqref{E:signu} and Cauchy-Schwarz. By  \eqref{E:nuDe} and Cauchy-Schwarz, the right-hand side is bounded above by 
	\begin{align*}
	c e^{dn^2} \lf(\int a^{4 \de^{-6} f^3} d \mu \rg)^{1/4} =  c e^{dn^2} \lf(\int a^{4 \de^{-6} f^3} \frac{d \mu}{d \nu} d \nu \rg)^{1/4} 
	\le ce^{dn^2}  \lf(\int a^{8 \de^{-6} f^3} d \mu \rg)^{1/8}. 
	\end{align*}
	The last inequality follows by Cauchy-Schwarz since $d\mu/d\nu \in L^2(\nu)$. Finally, by  \eqref{E:ae6} for all small enough $a > 1$ the right-hand side is at most $ce^{dn^2}$. This shows the first claim. The proof of the second claim is similar.
\end{proof}

\begin{proof}[Proof of Theorem \ref{t:geod-thm}] In the proof, $c, d > 0$ are positive constants that may change from line to line.
	Given Lemma \ref{L:easier} and the discussion before Lemma \ref{l:mod-cont-1} it suffices to prove the two bounds \eqref{E:translated1} and \eqref{E:translated2}. For  \eqref{E:translated1}, let $b> 0$ and let $X_n, Y_n$ be as in Lemma \ref{L:XY}. For $a > 1$ we have
	\begin{equation}
	\label{E:bound}
	a^{|X - Y|^3/\de^6} \le a^{ |X - Y|^3/\de^6} \mathbf{1}(|X - Y| > b) + \sum_{n \in \Z} a^{ |X_n - Y_n|^3/\de^6} \mathbf{1}((X, Y) = (X_n, Y_n)).
	\end{equation}
	By Lemma \ref{L:long-bound}, the expectation of the first term on the right side of \eqref{E:bound} is uniformly bounded over $\sigma \ge 1, \de \in (0, 1)$ for large enough $b$ and small enough $a>1$.
	
	By Cauchy-Schwarz we have
	$$
	\E \Big(a^{ |X_n - Y_n|^3/\de^6} \mathbf{1}((X, Y) = (X_n, Y_n))\Big)\le \Big(\E a^{2|X_n - Y_n|^3/\de^6}\P((X, Y) = (X_n, Y_n))\Big)^{\frac12}.
	$$
	By Corollary \ref{C:geod-locale}, we have
	$$
	\P((X, Y) = (X_n, Y_n))) \le ce^{-d |n|^3}
	$$ 
	for all $n, \sigma, \de$. Combining this with Lemma \ref{L:XY}, we get 
	$$
	\E \Big(a^{ |X_n - Y_n|^3/\de^6} \mathbf{1}((X, Y) = (X_n, Y_n))\Big)\le ce^{-d |n|^3+d'n^2}.
	$$
	Summing over $n$ shows that for small enough $a>1$ and large enough $b$, \eqref{E:bound} is uniformly bounded (in expectation) over $\sig \ge 1, \de \in (0, 1)$, completing the proof of \eqref{E:translated1}. The proof of \eqref{E:translated2} is similar. 
\end{proof}

\section{Variations}
\label{S:variation}
Finally, we prove the convergence of variations. Recall from the introduction that for any function $f:[s,s+t]\mapsto\mathbb R$, and $\alpha>0$, the $\alpha$-variation of $f$ on scale $\epsilon$ is defined as
$$
V_{\alpha,\epsilon}(f)=\sum_{u \in [s+ \epsilon,t]\,\cap \,\epsilon \mathbb Z}|f(u)-f(u-\epsilon)|^\alpha\,.
$$
Throughout this section, we let $\Ga$ denote the almost surely unique $\cL$-geodesic from $(B-R, 0)$ to $(-B -R, 1)$, where $B$ is a two-sided Brownian motion, $R$ is a two-sided Bessel-$3$ process, $\cL$ is a directed landscape, and all three objects are independent.

\begin{lemma}\label{l:firstlemma} Let $\pi$ be the geodesic from $(0,0)$ to $(0,1)$. Then for all $0 \le a<b\le1$,
	$$
	V_{3/2, \epsilon}(\pi|_{[a,b]})\to \nu(b-a)
	$$
	in probability as $\epsilon\to 0$,
	where $\nu=\E |\Ga(1)-\Ga(0)|^{3/2}$.
\end{lemma}

\begin{proof}
	Define a step function  $D_\e:[a, b] \to [0, \infty)$ as follows.
	\[
	D_\e(x)=\sum_{u \in [a+\e, b]\,\cap \,\epsilon \mathbb Z}\e^{-1}\E|\pi(u)-\pi(u-\epsilon)|^{3/2}\ind_{u-\e<x\le u}\,.
	\]
	By Theorem \ref{t:vimp'} and Theorem \ref{t:geod-thm}, the functions $D_\e$ are uniformly bounded and $D_\ep(x) \to \nu$ for all $x\in (a, b)$. Therefore by the bounded convergence theorem,
	\begin{equation}
	\label{E:V32}
	\E V_{3/2, \epsilon}(\pi|_{[a,b]}) = \int_a^b D_\e(x) dx \to (b-a)\nu \quad \text{ as } \e \to 0. \end{equation}
	Similarly define
	\[A_\e(x,y)=\sum_{u,v \in [a+\e, b]\,\cap \,\epsilon \mathbb Z}\e^{-2}\E(|\pi(u)-\pi(u-\epsilon)|^{3/2}|\pi(v)-\pi(v-\epsilon)|^{3/2})\ind_{u-\e<x\le u}\ind_{v-\e<y\le v}\,.\]
	By Theorems \ref{T:joint-cvg} and \ref{t:geod-thm}, the functions $A_\e$ are uniformly bounded and $A_\e(x, y) \to \nu^2$ pointwise for all $x\neq y$ and $x,y\in (a, b)$. Therefore by the bounded convergence theorem,
	\begin{equation}
	\label{E:Vvar}
	\E V_{3/2, \epsilon}^2(\pi|_{[a,b]}) = \int_{[a, b]^2} A_\e(x, y) dxdy \to \nu^2 \quad \text{ as } \e \to 0.
	\end{equation}
	Combining \eqref{E:V32} and \eqref{E:Vvar} gives that $\Var (V_{3/2, \epsilon}(\pi|_{[a, b]}))\to 0$ and the lemma follows.
\end{proof}
The proof of the following lemma is similar, and so we omit it. 
\begin{lemma}
	\label{l:weight}
	Let $W(\cdot)$ be the weight function along the geodesic $\pi$ from $(0,0)$ to $(0,1)$. Then for any $0\leq a<b\leq 1$,
	$$
	V_{3, \epsilon}(W|_{[a,b]})\to \mu(b-a)\,,
	$$
	in probability as $\epsilon\to 0$,
	where $\mu=\E|\ell(\Ga)|^3$.
\end{lemma}

The following corollary of Lemma \ref{l:firstlemma} and Lemma \ref{l:weight} is immediate.
\begin{corollary}\label{c:v32} There exist a sequence $\epsilon_n\to 0$ so that with probability 1 for all  $0\le s<t\le1$, we have
	$$
	\lim_{n \to \infty} V_{3/2, \epsilon_n}(\pi|_{[s,t]})=\nu(t-s)\,, \qquad \lim_{n \to \infty} V_{3, \epsilon_n}(W|_{[s,t]})=\mu(t-s)
	$$
	where $\nu=\E |\Ga(1) - \Ga(0)|^{3/2}$ and $\mu=\E|\ell(\Ga)|^3$.
\end{corollary}
\begin{proof} By Lemma \ref{l:firstlemma} and Lemma \ref{l:weight}, we can find a subsequence $\e_n \to 0$ such that the convergences above holds almost surely for all rational $s < t \in [0, 1]$. The extension to all real $s <t$ along this subsequence holds since 
	$$
	V_{\al, \e}(f|_{[s,t]}) \le V_{\al, \e}(f|_{[s',t']})
	$$
	for any $\al, \e, f$, and any intervals $[s, t] \sset [s', t']$.
\end{proof}

Finally, we can prove Theorem \ref{t:varintro}. We restate the theorem here for the reader's convenience.
\begin{theorem}
	\label{t:var-full}
	There exists a sequence $\epsilon_n\to 0$ so that with probability 1 the following holds. For all $u=(x,s';y,t')\in \mathbb R^4_\uparrow$ and any $[s,t]\subset(s',t')$, the leftmost geodesic $\pi_u$ from $(x,s')$ to $(y,t')$ and its weight function $W_u$ satisfy
	$$
	\lim_{n \to \infty} V_{3/2, \epsilon_n}(\pi_{u}|_{[s,t]}) = \nu(t-s),\,\qquad 	\lim_{n \to \infty} V_{3, \epsilon_n}(W_{u}|_{[s,t]}) = \mu(t-s).
	$$
\end{theorem}
\begin{proof} This holds by Corollary \ref{c:v32} for $u = (0, 0; 0, 1)$. The extension to all rational $u$ follows by the invariance properties of $\cL$ in Lemma \ref{l:invariance}.
	The extension to all real $u$ follows from Corollary \ref{C:geodesic-frame}.
\end{proof}

\section{Sharp upper bounds on H\"older exponents}
\label{s:holder-not!}

Proposition 12.3 in \citDOV, recorded as Proposition \ref{P:mod-dg} here, shows that the directed geodesic is almost surely H\"older-$\alpha$ for $\alpha<2/3$. This was also shown in Theorem $1.1$ of \citep{hammond2020modulus} for the prelimiting geodesics in Poissonian last passage percolation, a model in the KPZ universality class. Here we show that the directed geodesic is not H\"older-$2/3$.

In addition to the asymptotic independence of increments shown in Theorem \ref{T:joint-cvg}, the only other ingredient we need is that the limit of the increment law is unbounded. As in Section \ref{S:variation}, we let $\Ga$ denote the almost surely unique $\cL$-geodesic from $(B-R, 0)$ to $(-B -R, 1)$, where $B$ is a two-sided Brownian motion, $R$ is a two-sided Bessel-$3$ process, $\cL$ is a directed landscape, and all three objects are independent.

\begin{lemma}
	\label{l:unbounded-inc}
	The random variable $|\Ga(1) - \Ga(0)|$ is unbounded.
\end{lemma}

\begin{proof}
	We have
	$$
	(\Ga(0), \Ga(1)) = \argmax_{x, y} (C(x, y) + B(x) - B(y)),
	$$
	where $B$ is a two-sided Brownian motion and $C(x, y)$ is a random function, independent of $B$. Since the argmax above is attained, the function $C(x, y) + B(x) - B(y)$ is almost surely bounded above. Therefore for every $z> 0$, there exists a $c_z > 0$ such that
	\begin{equation}
	\label{e:C0t}
	\P \Big(C(z, 0) + B(z) - B(0) \ge -c_z,  C(x, y) + B(x) - B(y) \le c_z \text{ for all } x, y \in \R \Big) > 0.
	\end{equation}
	Now define
	$$
	f_z(x) = \begin{cases}
	4c_z - \frac{4c_z|x - z|}{z}, \qquad &x \in [0, 2z] \\
	0, \qquad &\text{ else.}
	\end{cases}
	$$
	The laws of $B - f_{ z}$ and $B$ are mutually absolutely continuous. Therefore \eqref{e:C0t} also holds when $B$ is replaced by $B - f_{ z}$, so with positive probability we have
	\begin{equation*}
	\begin{split}
	C(z, 0) + B(z) - B(0) &\ge -c_z + f_{z}(z) - f_{z}(0) = 3 c_z, \qquad \text{ and } \\
	C(x, y) + B(x) - B(y) &\le c_z + f_{z}(x) - f_{z}(y)\le c_z + \frac{4c_z|x - y|}z \qquad \text{ for all } x,y \in \R.
	\end{split}
	\end{equation*}
	The last inequality holds since the function $f_z$ is Lipschitz with constant $4c_z/z$.
	On this event, $c_z + 4 c_z |\Ga(1) - \Ga(0)|/z \ge 3 c_z$, so $|\Ga(1) - \Ga(0)| \ge z/2$. Since $z>0$ was arbitrary, this implies that $|\Ga(1) - \Ga(0)|$ is unbounded.
\end{proof}

\begin{theorem}\label{t:holder}
	Let $\pi:[0, 1] \to \R$ denote the geodesic from $(0, 0)$ to $(0, 1)$. Then almost surely, $\pi$ is not H\"older-$2/3$.
\end{theorem}

\begin{proof}
	Let $I(x, r) = r^{-2/3}|\pi(x-r) - \pi(x)|$. It suffices to show that $I := \sup_{x-r, x \in [0, 1]} I(x, r)$ is almost surely infinite. 
	
	Let $W_i$ be independent having the same distribution as $|\Gamma(1) - \Gamma(0)|$. By Lemma \ref{l:unbounded-inc}, for every $p<1$ and $m > 0$ there exists $k \in \N$ such that  
	$$
	\P(\max(W_1, \dots, W_k)>m) \ge p.
	$$
	By Theorem \ref{T:joint-cvg} and the continuous mapping theorem, we have
	$$
	I_k(r) := \max_{i \in \{1, \dots, k\}} I\left(\frac{i}{k+1}, r \right) \cvgd \max(W_1, \dots, W_k)
	$$
	as $r \to 0$, and so by the Portmanteau theorem, $\liminf_{r \to 0} \P(I_k(r) > m) \ge p$. Since $I \ge I_k(r)$ for all $k, r$, this implies that $\P(I > m) \ge p$. Since $m,p$ were arbitrary, $I$ is infinite almost surely.
\end{proof}

Similarly we show that the weight function is almost surely H\"older-$1/3^-$ but not H\"older-$1/3$. We first show a modulus of continuity that we expect is optimal.  

\begin{proposition}
	\label{P:mod-wt}
	Let $W$ be the weight function of the directed geodesic $\pi$ from $(0,0)$ to $(0,1)$. There exists $a>1$ and a random variable  $D$  with $\E a^{D^{3/2}}<\infty$  such that 
	$$
	|W(t) - W(s)| \le D|t-s|^{1/3} \log^{2/3}\lf( \frac{2}{|t-s|}\rg) \qquad \mbox{ for all }s,t\in [0,1].
	$$
\end{proposition}
\begin{proof}
	By the second part of Theorem \ref{t:geod-thm} and Markov's inequality, there are universal constants $c,d>0$ so that  $$\P[|W(t)-W(t+\e)|>\eps^{1/3} m]\le ce^{-dm^{3/2}}$$
	for all $t,t+\eps\in[0,1]$ and $m>0$. Lemma 3.3 in \citep{DV} converts such bounds into modulus of continuity bounds, and yields the result.  
\end{proof}

\begin{theorem}
	Let $\pi:[0, 1] \to \R$ denote the geodesic from $(0, 0)$ to $(0, 1)$ and let $W$ be its weight function. Then almost surely, $W$ is not H\"older-$1/3$.
\end{theorem}

\begin{proof} We first show that $\l(\Gamma)=\S(X,Y)$ is an unbounded random variable, where $X=\Gamma(0), Y=\Gamma(1)$. In this proof, $c,d_1,d_2$ will denote absolute positive constants whose values may change from line to line. Let $B$ be a two-sided Brownian motion. For all large enough $z>0$, $c_z=z^3$, and any small enough $\delta>0$ , we have
	\begin{eqnarray}\label{e:unbde1}
	\P\Big(|B(x)-B(y)|<(1+|x-y|^{1/2})(z+c\log(2+|x|+|y|)) \text{ for all } x,y\in \R,\nonumber\\ |B(x)|\le z^{1-2\delta} \text{ for all } x\in [-c_z,c_z]\Big)>0\,.
	\end{eqnarray}
	Indeed, by \eqref{e:eqq3} it follows that the probability of the first of the two events above is at least $1-e^{-d_1z^2}$, whereas the probability of the second is at least $e^{-d_2z^{1+4\delta}}$. So by choosing $\delta$ so that $1+4\delta<2$, we get that the above probability is positive using a simple union bound.
	
	Also for a two-sided Bessel-$3$ process $R$
	\begin{equation}\label{e:unbde2}
	\P\Big(R(x)\ge |x|^{1/2-\delta}-10 \text{ for all } x\in \R\Big)>0\,.
	\end{equation}
	This follows from the Brownian scaling of $R$, the fact that the future minimum of $R$ after time $1$ has a bounded density near $0$ and a union bound of the events $\{\min_{x\in [2^k,2^{k+1}]} R(x)\le x^{1/2-\delta}-10\}$ for $k=0,1,2,\ldots$
	
	Moreover, with $\S(x,y):=\cL(x,0;y,1)$, because of Lemma \ref{L:sheet-bd} and the upper tail bound of $\S(0, 0)$ (see Theorem \ref{T:TW-airy}), using a simple union bound we have for all large enough $z$,
	\begin{equation}\label{e:unbde3}
	\P\Big(\S(x,y)+(x-y)^2\le z+c\log(2+|x|+|y|), \S(0,0)>z^{1-\delta}\Big)>0\,
	\end{equation}
	for some absolute constant $c>0$.
	
	Let $A_z$ denote the intersection of the above three events in \eqref{e:unbde1}, \eqref{e:unbde2} and \eqref{e:unbde3} where $B,R$ and $\S$ are independent. Then $\P(A_z)>0$ for large enough $z$ and small enough $\de$. On the event $A_z$, for all $x,y$, 
	$\S(x,y)+B(x)-B(y)-R(x)-R(y)$ is at most
	\begin{equation}\label{e:upbndsxy}
	-|x|^{1/2-\delta}-|y|^{1/2-
		\delta}-(x-y)^2+z|x-y|^{1/2}+c\log(2+|x|+|y|)(1+|x-y|^{1/2})+20+2z\,,
	\end{equation}
	\[\text{ and }\qquad \S(0,0)\ge z^{1-\delta}\,.\]
	Now, $z|x - y|^{1/2} \le \frac{3}{4}z^{4/3} + \frac{1}{4}|x-y|^2$ for all $x, y \in \R$ and $z > 0$. Also,
	for large enough $z$ and any $x, y \in \R$,
	$$
	c\log(2+|x|+|y|)(1+|x-y|^{1/2})+20 + 2z\le |x -y|^2/2 + z^{4/3}.
	$$
	Therefore for large enough $z$, the sum in \eqref{e:upbndsxy} is at most
	\[ -|x|^{1/2-\delta}-|y|^{1/2-
		\delta}+2z^{4/3}\,.\]
	Thus if either $|x|$ or $|y|$ is at least $z^3$, $\delta$ is small enough so that $1/2-\delta>4/9$, and $z$ is large enough, then we get that $\max(|x|^{1/2-\delta}, |y|^{1/2-\delta}) \ge 2z^{4/3}$ and hence the quantity in \eqref{e:upbndsxy} is negative.
	
	Thus for $c_z=z^3$ and $K_z:=[-c_z,c_z]^2$, on the event $A_z$,
	\[\S(0,0)>\S(x,y)+B(x)-B(y)-R(x)-R(y) \qquad \text {for all } (x,y) \in K_z^c\,,\]
	that is, for $(X,Y)=(\Gamma (0),\Gamma(1))$, where $\Ga$ denotes the almost surely unique $\cL$-geodesic from $(B-R, 0)$ to $(-B -R, 1)$, we have $(X,Y)\in K_z$. As $(X,Y)$ maximizes $\S(x,y)+B(x)-B(y)-R(x)-R(y)$ and $R(x)\ge 0$ for all $x$, on the event $A_z$,
	\[\S(X,Y)\ge \S(0,0)+R(X)+R(Y)-B(X)+B(Y)\ge \S(0,0)-2\sup_{x\in [-c_z,c_z]}|B(x)|\,.\]
	As the event $A_z$ implies $\S(0,0)\ge z^{1-\delta}$ and $\sup_{x\in [-c_z,c_z]}|B(x)|\le z^{1-2\delta}$, on the event $A_z$,
	\[\S(X,Y)\ge cz^{1-\delta}\,.\]
	Since $z>0$ was arbitrary, this implies that $\l(\Gamma)=\S(X,Y)$ is an unbounded random variable. Hence, as in the proof of Theorem \ref{t:holder}, using Theorem \ref{T:joint-cvg} we have the result.
\end{proof}

\section{Open problems}

We end with a collection of open problems about geodesics and their weight functions that arise naturally from our work. Let $\pi$ be the directed geodesic from $(0,0)$ to $(0, 1)$, and let $W$ be its weight function. 

\begin{problem}
	Prove that the modulus of continuity for the directed geodesic $\pi$, Theorem 1.7 in \citep{DV} is sharp in the following sense
	$$
	\lim_{\e\to 0}\;\sup_{\stackrel{s\not=t\in[0,1]}{|s-t|<\e}}\;	\frac{|\pi(t) - \pi(s)|} { |t-s|^{2/3} \log^{1/3}\lf(2/|t-s|\rg)}>0 \qquad \mbox{ a.s.} 
	$$
\end{problem}

\begin{problem}
	Prove that the modulus of continuity for the weight function along $W$ along a geodesic, Proposition \ref{P:mod-wt}, is sharp in the following sense:
	$$
	\lim_{\e\to 0}\;\sup_{\stackrel{s\not=t\in[0,1]}{|s-t|<\e}}\;	
	\frac{|W(t) - W(s)|} { |t-s|^{1/3} \log^{2/3}\lf(2/|t-s|\rg)}>0 \qquad \mbox{ a.s.} 
	$$
\end{problem}

Let $\Gamma$ be the limiting rescaled geodesic increment as in Theorem \ref{t:varintro}, and recall that  $\nu=\E|\Gamma(1)-\Gamma(0)|^{3/2}$ and $\mu=\E |\ell(\Gamma)|^3$ be as in Theorem \ref{t:varintro}. 

\begin{problem}
	Let $\eps_n\to 0$ along a fixed sequence. Show that  $V_{3/2,\mathcal \eps_n}(\pi)\to \nu$ a.s.
\end{problem}
The analogous problem for $W$ is also open. 
\begin{problem}
	Express the quantities $\nu$ or $\mu$ explicitly (possibly using Fredholm determinants.) 
\end{problem}
In fact, we do not know much about the distribution of $\Gamma(1)-\Gamma(0)$ or $\ell(\Gamma)$. They both have mean zero by symmetry and scale-invariance, respectively. It would be nice to have explicit formulas for their distribution, or even their variance. 

It follows from the limiting procedure, Theorem \ref{t:defgamma}, that the processes $\Gamma_n(t)=n^2\big(\Gamma(1/2+t/n^3)-\Gamma(1/2)\big)$ defined on $[-n^3/2,n^3/2]$ form a consistent family, and so they have a distributional limit $\Gamma_{\infty}:\R\to \R$ when $n\to\infty$. The analogous statement holds for the entire environments defined in Theorem \ref{t:defgamma}. $\Gamma_{\infty}$ is the {\bf bi-infinite geodesic} started at 0. It is scale invariant, $\sigma^2 \Gamma_{\infty}(\cdot/\sigma^3)\eqd \Gamma_{\infty}$, and has stationary increments,  $\Gamma_{\infty}(\cdot +a)-\Gamma_{\infty}(a)\eqd \Gamma_{\infty}$, but it does not have independent increments. 
\begin{problem}\label{op:5}
	Describe the distribution of $\Gamma_\infty$ directly, without the directed landscape. 
\end{problem}

\bibliography{3halfcitation}
\bibliographystyle{amsplain}

\end{document}